\documentclass{amsart}

\usepackage{hyperref, color}
\usepackage{mathrsfs}
\usepackage{paralist}
\usepackage{cancel}
\usepackage{graphicx}
\usepackage{subfig}

\newtheorem{theorem}{Theorem}[section]

\newtheorem{proposition}[theorem]{Proposition}
\newtheorem{definition}[theorem]{Definition}
\newtheorem{corollary}[theorem]{Corollary}

\newtheorem{remark}[theorem]{Remark}

\newfont{\bb}{msbm10 at 10pt}
\newfont{\bbp}{msbm10 at 8pt}

\def\b{\hbox{\bb B}}

\def\r{\hbox{\bb R}}
\def\R{\hbox{\bb R}}
\def\h{\hbox{\bb H}}

\def\s{\hbox{\bb S}}
\def\S{\hbox{\bb S}}

\def\man{\mathcal{M}}
\newcommand{\metL}[2]{\ll #1,#2 \gg }
\newcommand{\meta}[2]{\langle #1,#2 \rangle }

\newcommand{\abs}[1]{\left\vert #1 \right\vert}
\newcommand{\set}[1]{\left\{#1\right\}}
\newcommand{\ov}[1]{\overline{#1}}

\newcommand{\De}{\Delta}

\begin{document}

\title[Escobar's type Theorems]{Escobar's Type Theorems for elliptic fully nonlinear degenerate equations}

\author[Abanto]{D. P. Abanto}
\address{Instituto Nacional de Matem\'atica Pura e Aplicada, Rio de Janeiro - Brazil}
\email{dimas@impa.br}

 \author[Espinar]{J.  M. Espinar}
 \address{Instituto Nacional de Matem\'atica Pura e Aplicada, Rio de Janeiro - Brazil}
 \email{jespinar@impa}

\thanks{The authors were partially supported by CNPq-Brazil. The second author is partially supported is partially supported by Spanish MEC-FEDER Grant
MTM2013-43970-P; CNPq-Brazil Grants 405732/2013-9 and 14/2012 - Universal, Grant
302669/2011-6 - Produtividade; FAPERJ Grant 25/2014 - Jovem Cientista de Nosso Estado.}

\subjclass[2010]{Primary 53Cxx, 58Jxx; Secondary 35Pxx.}

\date{\today}


\keywords{Conformal equations, Yamabe problem with boundary, Escobar Problem,  hyperbolic gauss map}

\begin{abstract}
{In this paper we prove non-existence and classification results for elliptic fully nonlinear elliptic degenerate conformal equations on certain subdomains of the sphere with prescribed constant mean curvature along its boundary. We also consider non-degenerate equations. Such subdomains are the hemisphere (or a geodesic ball in $\mathbb{S}^m$), punctured balls and annular domains.

Our results extend those of Escobar in \cite{Esc} when $m\geq 3$, and Hang-Wang in \cite{HaWa} and Jimenez in \cite{J} when $m=2$.}
\end{abstract}

\maketitle

\section{Introduction}\label{intro}

Let $(\man^n, g_0)$ be a compact orientable Riemannian manifold with smooth boundary and dimension $n\geq 3$. Let us denote by  $R(g_0)$  its scalar curvature and by $h(g_0)$ its boundary mean curvature with respect to the outward unit normal vector field. If $g=u^{\frac{4}{n-2}}g_0$ is a  metric  conformal to $g_0$ then its scalar curvature and
boundary mean curvature  are related by the following nonlinear elliptic partial differential equation of  critical Sobolev exponent in terms of the positive function $u$
\begin{equation}\label{pde}
   \left\{ \begin{array}{lrc}
  \De_{g_0} u - \frac{n-2}{4(n-1)}R(g_0) u +
  \frac{n-2}{4(n-1)}R(g)u^{\frac{n+2}{n-2}}=0 &  \textrm{ in } \man, \\ \\
  \frac{\partial u}{\partial \eta}+\frac{n-2}{2}h(g_0) u -
  \frac{n-2}{2}h(g)u^{\frac{n}{n-2}}=0 & \textrm{ on } \partial \man ,
   \end{array}  \right.
  \end{equation}
where $\Delta_{g_0}$ is the Laplacian with respect to the metric $g_0$ and $\eta$ is the outward unit normal vector field along $\partial \man$. The problem of existence of solutions of (\ref{pde}) when $R(g)$ and $h(g)$ are constants is referred as the \index{Yamabe problem}\emph{Yamabe problem} which was completely solved when $\partial \man = \emptyset$ in a sequence of works, beginning with H. Yamabe himself ~\cite{Yam}, followed by N. Trudinger ~\cite{Tru} and T. Aubin ~\cite{Aub}, and finally by R. Schoen ~\cite{Scho}.

When the manifold $(\man ,g_0)$ is complete but not compact, the existence of a conformal metric solving the Yamabe Problem does not hold in general, as we can see in the work of Zhiren \cite{Zhi}.

In the case $\man$ is compact with nonempty  boundary  almost all the cases were solved by the works of J. Escobar  ~\cite{Esc,Esc2,Esc3,Esc4},  continued by F. Marques ~\cite{Ma} among others. We will also refer to this problem as the {\it Escobar Problem}.

By far the most important is the case where $\man$ is the closed unit Euclidean ball or equivalently,  the closed hemisphere $\ov{\mathbb S^n_+}$ endowed with the standard round metric $g_0$ and $R(g)$ is a constant. Regarding to the existence of solutions, J. Escobar proved the following

\begin{quote}
{\bf Theorem \cite{Esc4}.} {\it Let $ \Omega \subset \r^{m}$, $m > 6$, be a bounded domain with smooth boundary. There exists a smooth metric g conformally related to the Euclidean metric such that the scalar curvature of g is zero and the mean curvature of the boundary with respect to the metric g is (positive) constant.}
\end{quote}

Also, J. Escobar proved:

\begin{quote}
{\bf Theorem \cite{Esc2}.} {\it Any bounded domain in the Euclidean space $\r^{m}$, with smooth boundary and $m \geq 3$, admits a metric conformal to the Euclidean metric having (non-zero) constant scalar curvature and minimal boundary.}
  \end{quote}

Regarding to the classification of solutions to the Escobar Problem, in the case that $(\man ,g_0)$ is the closed Euclidean ball $\overline{\b ^{m}}$, J. Escobar showed that the solution to the Yamabe Problem with Boundary must have constant sectional curvatures. Also, he proved that the space of solutions in the Euclidean ball is empty when the scalar curvature is zero and the mean curvature is a non-positive constant \cite{Esc}.

The existence of solutions to the Yamabe problem on non-compact manifolds $(\man,g_{0})$ with compact boundary was proved for a large class of manifolds in the work of F. Schwartz \cite{Schwa}. He proved that the Riemannian manifolds that are \textit{positive} and their ends are \textit{large} have a conformal metric of zero scalar curvature and constant mean curvature on its boundary (see \cite{Schwa} for details). Even more, F. Schwartz proved the following:
\begin{quote}
{\bf Theorem \cite{Schwa}.} {\it
Any smooth function $f$ on $\partial \man$ can be realized as the mean curvature of a complete scalar flat metric conformal to $g_{0}$.}
\end{quote}

The Yamabe Problem opened the door to a rich subject in the last few years: \emph{the study of conformally invariant equations}. More precisely, \emph{given a smooth functional $f(x_{1} , \ldots , x_{m})$, does there exist a conformal metric $g = e^{2\rho}g_0$ in $\man$ such that the eigenvalues $\lambda _i$ of its Schouten tensor satisfy
$ f(\lambda _{1} , \ldots , \lambda _{m} ) = c  \, \text{ in } \man ?$}

Given a Riemannian manifold $(\man ^m, g)$, $m \geq 3$, the Schouten tensor of $g$ is given by
$$ {\rm Sch}(g) := \frac{1}{m-2}\left( {\rm Ric}(g) -
\frac{R(g)}{2(m-1)}g\right)$$where ${\rm Ric}(g)$ and $ R(g)$ are the Ricci tensor and the scalar curvature function of $g$ respectively.

Note that $f(x_{1} , \ldots , x_{m} ) = x_{1} +\cdots + x_{m}$ reduces to the Yamabe Problem. It is of special interest to consider $f(\lambda )\equiv \sigma _k (\lambda )^{1/k}$, $\lambda=(\lambda_{1},\ldots,\lambda_{m})$, where $\sigma _k (\lambda )$ is the $k-$th elementary symmetric polynomial of its arguments $\lambda_{1},\ldots,\lambda_{m}$ and set it to be a constant, i.e.,
$ \sigma _k (\lambda ) = {\rm constant}$, such problem is known as the $\sigma _k - $Yamabe Problem. This is an active research topic and has interactions with other fields as Mathematical General Relativity \cite{Bra,Gurs}.

Interesting problems arise in this context of conformally invariant equations. One of them is the classification of complete conformal metrics satisfying a Yamabe type equation on a subdomain of the sphere, in the line of  Y.Y. Li and collaborators \cite{YLi06,YLi,LiLi1,LiLi2,LiLi3,YLiLNgu14}. Also, it is interesting to find non-trivial solutions of conformal metrics on subdomains of the sphere prescribing the scalar curvature in the interior, or other elliptic combination of the Schouten tensor, and the mean curvature of the boundary. Such problem is related to the Min-Oo conjecture when we consider the scalar curvature inside. S. Brendle, F.C. Marques and A. Neves \cite{SBreFMarANev10} showed the existence of such non-trivial metric in the hemisphere, however such metric is not conformal to the standard one. In other words, could one find conditions on the interior and the boundary that imply that such conformal metric is unique (see \cite{YLiLNgu14,Spie})? In this work we will focus in the case $\man $ is a subdomain of the $m-$dimensional sphere $\s^{m}$.

Let us explain in more detail the meaning of a fully non-linear conformally invariant elliptic equation. Originally, these type of equations are second order elliptic partial differential equations in $\r^{m}$. The problem is to find a function $u>0$ satisfying an identity of the type
\begin{equation*}
  \mathcal{F}(\cdot,u,\nabla u, \nabla^{2} u)=c,
\end{equation*}where $c$ is a constant.

Such kind of equation is called {\it conformally invariant} if for all M\"{o}bius transformation $\psi$ in $\r^{m}$ and any positive function $u\in C^{2}\left( \r^{m} \right)$, it holds
\begin{equation}\label{ref}
  \mathcal{F}\left(\cdot,u_{\psi},\nabla u_{\psi},\nabla^{2} u_{\psi} \right)=\mathcal{F}\left( \cdot,u,\nabla u,\nabla^{2} u \right)\circ\psi,
\end{equation}
where $u_{\psi}$ is defined by
\begin{equation*}
  u_{\psi}:=\left|J\psi \right|^{\frac{m-2}{2m}}u\circ\psi,
\end{equation*}and $J\psi$ is the Jacobian of $\psi$. For more details see \cite{LiLi1}.

One can check that if there is a smooth positive function $u:\mathbb{R}^{m}\to\mathbb{R}$ such that
\begin{equation*}
  \mathcal{F}\left( \cdot,u,\nabla u,\nabla^{2} u \right)=c,
\end{equation*}
then, from (\ref{ref}), we have that
\begin{equation*}
  \mathcal{F}\left(\cdot,u_{\psi},\nabla u_{\psi},\nabla^{2} u_{\psi} \right)=c
\end{equation*}for any M\"{o}bius transformation in $\r^{m}$.

Aobing Li and YanYan Li proved a fundamental relation between solutions of this type of equations and the eigenvalues of the Schouten tensor of a conformal metric related to such solution. Specifically:
\begin{quote}
{\bf Theorem \cite{LiLi1}.} {\it Let $\mathcal{F}(\cdot,u,\nabla u,\nabla^{2}u)$ be conformally invariant on $\r^{m}$. Then
  \begin{equation*}
  \mathcal{F}\left( \cdot,u,\nabla u,\nabla^{2} u \right)= \mathcal{F}\left( 0,1,0,-\frac{m-2}{2}A^{u} \right),
    \end{equation*}
    where
    \begin{equation*}
    A^{u} := -\frac{2}{m-2}u^{-\frac{m+2}{m-2}}\nabla^{2} u+ \frac{2m}{(m-2)^{2}} u^{-\frac{ 2m}{m-2}}\nabla u\otimes \nabla u -\frac{2}{(m-2)^{2}} u^{- \frac{2m}{m-2} }\left|\nabla u\right|^{2}I
    \end{equation*}
     and I is the $m\times m$ identity matrix. Moreover, $\mathcal{F}(0,1,0,\cdot)$ is invariant under orthogonal conjugation, i.e.,
     \begin{equation*}
     \mathcal{F}\left( 0,1,0,-\frac{m-2}{2} O^{-1}AO \right)=\mathcal{F}\left( 0,1,0,-\frac{m-2}{2} A \right) \quad \forall A\in\mathcal{S}^{m\times m},\ O\in O(m),
     \end{equation*}}
     where $\mathcal{S}^{m\times m}$ is the set of $m\times m$ symmetric matrices .
\end{quote}

Thus the behavior of $\mathcal{F}\left( \cdot,u,\nabla u,\nabla^{2} u \right)$ is determined by the matrix $A^{u}$, such matrix is nothing but the Schouten tensor of the conformal metric $g= u^{\frac{4}{n-2}}g_{Eucl}$. Then, in order to define a {\it conformally invariant equation}, we use functions $F\in C^{1}\left( U\right)\cap \in C^{0}\left( \ov{U} \right)$, where $U$ is an open subset of $\mathcal{S}^{m\times m}$, such that the following conditions hold:
\begin{enumerate}
  \item for all $O\in O(m)$: $O^{-1}AO\in U$ for all $A\in U$,
  \item for all $t>0$: $tA\in U$ for all $A\in U$,
  \item for all $P\in\mathcal{P}$: $P\in U$, where $\mathcal{P}\subset\mathcal{S}^{m\times m}$ is the set of $m\times m$ positive definite symmetric matrices,
  \item for all $P\in \mathcal{P}$: $A+P \in U$ for all $A\in U$,
  \item $0\in\partial U$.
\end{enumerate}

Also, the second order differential equation will be {\it elliptic} if the function $F$ satisfies
\begin{enumerate}
  \item for all $O\in O(m)$: $F\left(O^{-1}AO\right)=F(A)$ for all $A\in U$,
  \item $F>0$ in $U$,
  \item $F|_{\partial U}=0$,
  \item for every $M\in U$:
  \begin{equation*}
    \left( \frac{\partial F}{\partial M_{ij}}\right)\in \mathcal{P}.
  \end{equation*}
\end{enumerate}

The above conditions on $(F,U)$ allow us to simplify the function to a functional acting on the eigenvalues of the Schouten tensor, i.e., on the eigenvalues of $A^u$. In order to make this explicit, let us define the following subsets:
\begin{align*}
  \Gamma_{m} =& \{x\in\r^{m}: x_{i}>0, i=1,\ldots,m\}, \\
  \Gamma_{1}= & \left\{ x\in\r^{m} : x_{1}+\cdots+x_{m}>0\right\} .
\end{align*}

Let $\Gamma\subset\r^{m}$ be a symmetric open convex cone and  $f\in C^{1}\left(\Gamma\right)\cap C^0\left(\ov{\Gamma}\right)$ such that
\begin{enumerate}
  \item $ \Gamma_{m} \subset \Gamma \subset \Gamma_{1} $,
  \item $f$ is symmetric,
  \item $f>0$ in $\Gamma$,
  \item $f|_{\partial\Gamma}=0$,
  \item $f$ is homogeneous of degree 1,
  \item for all $x\in\Gamma$ it holds $\nabla f(x)\in\Gamma_{m}$.
\end{enumerate}

Now, we will see how to obtain the open set $U\subset\r^{m}$ and the function $F:\ov{U}\to\r$ satisfying the above properties from the data $\Gamma $ and $f$. The pair $(f, \Gamma)$ is called {\it elliptic data}. From $f:\ov{\Gamma}\to\mathbb{R}$ we define
\begin{equation*}
  U=\left\{ A\in\mathcal{S}^{m\times m}: \lambda(A)\in\Gamma \right\},
\end{equation*} where $\lambda(A)=(\lambda_{1},\ldots,\lambda_{m})$ are the eigenvalues of $A$. Since $\Gamma$ is symmetric it is well defined. Also we define
\begin{equation*}
  F(A)=f(\lambda(A)).
\end{equation*}

Observe that the function $F:U\to\mathbb{R}$ is in $C^{1}(U)$ and it can be continuously extended to $\ov{U}$ such that $F|_{\partial U}=0$. Then, this function $F:\ov{U}\to\r$ and the set $U$ satisfy the properties listed above.

Hence, the problem with elliptic data $(f, \Gamma )$ for conformal metrics in a domain $\Omega\subset\s^{m}$ is to find a conformal metric $g=e^{2\rho}g_{0}$ to the standard metric $g_{0}$ such that
\begin{equation*}
  f(\lambda (g))=c \quad\text{in}\quad\Omega ,
\end{equation*}where $\lambda (g)=(\lambda_1,\ldots, \lambda_m )$ is composed by the eigenvalues of the Schouten tensor of $g=e^{2\rho}g_{0}$ and $c$ is a constant. We can also distinguish two cases, when $c>0$ is a positive constant, without loss of generality we can consider $c=1$, the problem is called \textit{non-degenerate}. When the constant satisfies $c=0$, the problem is called \textit{degenerate}.

By the stereographic projection, any domain of $\r ^m$ corresponds to a domain in $\s ^m$. Moreover, the stereographic projection is conformal, hence, any conformal equation in a domain of $\r ^m$ can be seen as a conformal equation in the corresponding domain in $\s ^m$, and vice-versa. Therefore, henceforth we will consider conformally invariant equations in subdomains of the sphere $(\s ^m , g_0)$ endowed with its standard metric.

Now, take $g=e^{2\rho}g_0$  in  $\Omega \subseteq \mathbb{S}^m$. The Yamabe problem for $R(g)=1$  and $h(g)=c$, where $h(g)$ is the boundary mean curvature with respect to the outward unit normal vector field, is equivalent to find a smooth function $\rho$ on $\Omega$ such that

\begin{equation}\label{yamabe2}
\left\{ \begin{array}{lcl}
  \lambda_1 + \dots + \lambda_m = \frac{1}{2(m-1)} ,  &\text{ in }  &\Omega , \\
h(g)=c , & \text{ on } & \partial \Omega.
\end{array}\right.
\end{equation}

Posed in this form, (\ref{yamabe2}) can be generalized to other functions of the eigenvalues of the Schouten tensor. For instance, one may consider the $\sigma_k$-Yamabe problem on $\mathbb S^m_+$  considering the  $k$-symmetric function of the eigenvalues of the Schouten tensor \cite{Sheng,Sheng2}. In this work we are interested in the fully nonlinear case of this problem, in the line opened by A. Li and Y.Y. Li \cite{LiLi1,LiLi2,LiLi3}. Namely, given $(f,\Gamma)$  an elliptic data and, $b \geq 0$ and  $c\in \r$, find $\rho \in C^{2,\alpha} (\overline {\mathbb{S}^m_+})$ so that $g=e^{2\rho}g_0 $ is a solution of
\begin{equation}\label{p1}
\left\{\begin{array}{rl}
f(\lambda(g))= b ,& \lambda(g)\in\Gamma \text{ in } \mathbb{S}^m_+,\\ \\

h(g) = c	, & \textrm{ on } \partial \mathbb{S}^m_+ .
\end{array} \right.
\end{equation}

M.P. Cavalcante and J.M. Espinar \cite{CE} have shown by geometric methods that
\begin{quote}
{\bf Theorem \cite{CE}.} {\it If $g=e^{2\rho}g_{0}$ is a conformal metric in $\overline{\s^{m}_{+}}$ that satisfies
  \begin{equation*}
    \left\{
     \begin{array}{rll}
         f(\lambda (g))&=1, & \text{in}\quad\s^{m}_{+}, \\
         h(g)&=c ,& \text{on}\quad\partial \s^{m}_{+},
     \end{array}
    \right.
  \end{equation*}then, there is a conformal diffeomorphism $\Phi:\s^{m}\to\s^{m}$, preserving $\s^{m}_{+}$, such that $g=\Phi^{*}\left( g_{0}   {\big|}_{\overline{\s^{m}_{+}}}\right)$.}
\end{quote}

Using analytic methods, A. Li and Y.Y. Li \cite{LiLi3} proved the result above.
Nevertheless, M.P. Cavalcante and J.M. Espinar went further and they dealt with annular domains, as J. Escobar did \cite{Esc} for the scalar curvature, in the fully nonlinear elliptic case. Let us denote by ${\bf n} \in \mathbb{S}^m_+ \subset \mathbb{S}^m$ the north pole and let $r < \pi/2$. Denote by $ B_{r} ( {\bf n} ) $ the geodesic ball in $\mathbb{S}^n$ centered at ${\bf n} $ of radius $r$. Note that, by the choice of $r$, $\partial \mathbb{S}^m_+ \cap \partial B_{r} ( {\bf n} ) = \emptyset$.

Denote by $\mathbb A (r)= \mathbb{S}^m_+ \setminus \overline{ B_{r} ( {\bf n} )}$ the annular region determined by $\mathbb{S}^m_+$ and $ B_{r} ({\bf n} )$. Note that the mean curvature of $\partial  B_{r} ({\bf n}) $ with respect to $g_0$ and the inward orientation along $\partial \mathbb A (r) $ is a constant $h(r)$ depending only on $r$. Let us consider the problem of finding a conformal metric in $\mathbb A (r)$ satisfying an elliptic condition in the interior and whose boundary components, $\partial  B_{r} ( {\bf n} ) $ and $\partial \mathbb{S}^m_+$, are minimal.

In other words, given $(f, \Gamma)$ an elliptic data, find $\rho \in C^\infty (\mathbb A (r))$ so that the metric $g=e^{2\rho}g_0 $ satisfies
\begin{equation}\label{p2}
\left\{\begin{array}{rl}
f(\lambda(g))= 1,& \quad\text{in}\quad \mathbb{A}(r),\\ \\

h(g)= 0, & \quad\textrm{on}\quad \partial B_{r} ({\bf n} )\cup \partial \mathbb{S}^m_+ .
\end{array} \right.
\end{equation}

In the above situation, M.P. Cavalcante and J.M. Espinar obtained:
\begin{quote}
{\bf Theorem \cite{CE}.} {\it Let $\rho \in C^{2,\alpha} (\overline{\mathbb{A}(r)})$ be a solution of \eqref{p2}. Then, $g= e^{2\rho}g_0$ is rotationally symmetric metric in $\overline{\mathbb{A}(r)}$. }
\end{quote}

This work is organized as follows. To facilitate access, the sections are rendered as self-contained as possible. In Section \ref{SectLocal}, we first review the local relationship between horospherically concave hypersurfaces in $\h^{m+1}$ and conformal metrics in $\s^{m}$. We begin by giving the definition of the Hyperbolic Gauss map for an oriented immersed hypersurface in $\mathbb{H}^{m+1}$. In such definition, we use the boundary at infinity of the Hyperbolic space, also called ideal boundary of $\mathbb{H}^{m+1}$, that is, the sphere $\s^{m}$. There are sufficient and necessary conditions for the hyperbolic Gauss map to be a local diffeomorphism. One of these conditions is related to the regularity of the light cone map of an oriented hypersurface. Other conditions are related to the principal curvatures of the given oriented hypersurface. Then, we define one of the important objects in our study, horospherically concave hypersurfaces in $\mathbb{H}^{m+1}$. These hypersurfaces are oriented and they have the property that its Hyperbolic Gauss map is a local diffeomorphism. The importance of this class of hypersurfaces is that, locally, we can give a conformal metric over the image of the Hyperbolic Gauss map (conformal to the standard metric $g_{0}$ in the sphere $\s^{m}$). Suppose that $g=e^{2\rho}g_{0}$ is this conformal metric, $\rho\in C^{2,\alpha}\left( \Omega\right)$, where $\Omega$ is a small open set that is contained in the image of the Hyperbolic Gauss map, the function $\rho$ has a geometric interpretation that is related to tangent horospheres to the original hypersurface. In the Poincar\'{e} ball model, $\rho$ is the signed hyperbolic distance between the tangent horosphere and the origin of the Poincar\'{e} ball model.

We have to recall now the Local Representation Theorem:

\begin{quote}
{\bf Local Representation Theorem \cite{EGM}.} {\it Let $\phi: \Omega \subseteq \s ^{m}\to \h^{m+1}$ be a piece of horospherically concave hypersurface with Gauss map $G(x)=x$. Then, it holds

\begin{equation*}
\phi = \frac{e^{\rho}}{2}\left( 1+ e^{-2\rho} \left( 1+ |\nabla\rho|^2 \right)\right) (1,x) + e^{-\rho} (0, -x +\nabla\rho).
\end{equation*}

Moreover, the eigenvalues $\lambda_i$ of the Schouten tensor of the horospherical metric $g = e^{2\rho}g_{0}$ and the principal curvatures $\kappa_i$ of $\phi$ are related by
\begin{equation*}
\lambda _i = \frac{1}{2} -\frac{1}{1+ \kappa _i} .
\end{equation*}

Conversely, given a conformal metric $g= e^{2\rho} g_{0}$ defined on a domain of the sphere $\Omega \subseteq \s ^{m}$ such that the eigenvalues of its Schouten tensor are all less than $1/2$, the map $\phi$ given above defines an immersed, horospherically concave hypersurface in $\h ^{m+1}$ whose Gauss map is $G(x)=x$ for $x\in \Omega$ and whose horospherical metric is the given metric $ g$.}
\end{quote}

\begin{remark}
In the above Local Representation Theorem we are using the Hyperboloid Model for $\h ^{m+1}$. However, we will use along this work other models for $\h ^{m+1}$ as the Poincar\'{e} Model or the Klein Model.
\end{remark}

By the Local Representation Theorem, one can see that the function $\rho$ is all we need to recover the original hypersurface. Such theorem is of great importance because we can obtain horospherically concave hypersurfaces with injective Gauss map from conformal metrics defined in domains $\Omega$ of the sphere $\s^{m}$ if we impose certain conditions. Such conformal metric is called the horospherical metric of the horospherically concave hypersurface in $\mathbb{H}^{m+1}$.

Hence, given a subdomain $\Omega\subset \mathbb{S}^{m}$ and $\rho\in C^{2,\alpha}\left(\Omega\right)$, consider the conformal metric $g=e^{2\rho}g_{0}$, then the question is: what can we say about the hypersurface given by the representation formula? It is known (cf. \cite{BEQ,EGM}) that if we impose certain conditions on the given conformal metric $g=e^{2\rho}g_{0}$, we realize a horospherically concave hypersurface with injective map Gauss. Our first result says that such horospherically concave hypersurface is proper. Specifically,

\begin{quote}
{\bf Theorem \ref{Ch3:proper}.} {\it Given $\rho\in C^{1}\left( \Omega \right)$, the map $\phi:\Omega\to\mathbb{H}^{m+1}$ is proper if, and only if, $ \left|\rho \right|_{1,\infty}(x)\to\infty$ when $x\to p$ for every $p\in\partial\Omega$.}
\end{quote}

Using this theorem, we can give a condition on a complete conformal metric that guaranties that the associated map is proper.

\begin{quote}\label{pro20}
\textbf{Theorem \ref{Ch3:Theo3}.} \textit{Let $g=e^{2\rho}g_{0}$ be a complete metric in $\Omega$, such that $\sigma=e^{-\rho}$ is the restriction of a continuous function that is defined in $\overline{\Omega}$. Then $\phi:\Omega\to\mathbb{H}^{m+1}$ is a proper map.}
\end{quote}

In the following, we make use of the parallel flow of a horospherically concave hypersurface, this flow is defined using the opposite to the canonical orientation of the hypersurface. More precisely, let $\eta$ be the canonical orientation along $\phi$, then for every $t>0$, we define the map $\phi_{t}:\Omega\to\h^{m+1}$ as
\begin{equation*}
  \phi_{t}(x)=\gamma\left( t,\phi(x),-\eta(x) \right)\quad\forall x\in\Omega,
\end{equation*}
where $\gamma\left( \cdot,\phi(x),-\eta(x) \right)$ is the geodesic in the Hyperbolic space $\h^{m+1}$ passing through $\phi(x)$ and has velocity $-\eta(x)$ at that point.

In fact, the map $\phi_{t}$ is a horospherically concave hypersurface in the Hyperbolic space $\h^{m+1}$ for every $t>0$. The horospherical metric of $\phi_{t}:\Omega\to\h^{m+1}$ is the conformal metric $g_{t}=e^{2t}g$, where $g$ is the horospherical metric of $\phi$. It is remarkable that the property of properness is invariant under the parallel flow (cf. Proposition \ref{Ch3:Prop5}).

Using the Local Representation Theorem we can say that the horospherically concave hypersurfaces that we get using the parallel flow correspond to dilations of the horospherical metric of the original horospherically concave hypersurface.

Another important issue about horospherically concave hypersurfaces in $\h ^{m+1}$ is embeddedness. We will see that if we impose some extra conditions on the conformal metric, then we get embeddedness along the parallel flow.

\begin{quote}
{\bf Theorem \ref{Ch3:theo33}.} {\it Let $\rho\in C^{2,\alpha}\left(\Omega\cup \mathcal{V}_{1}\right)$ be such that $\sigma=e^{-\rho}\in C^{2,\alpha}\left( \Omega\cup \mathcal{V}_{1}\right)$ satisfies:
   \begin{enumerate}
     \item $\sigma\cdot\sigma$ can be extended to a $C^{1,1}$ function on $\overline{\Omega}$.
     \item $\left< \nabla\sigma,\nabla\sigma \right>$ can be extended to a Lipschitz function on $\overline{\Omega}$.
   \end{enumerate}

Then, there is $t_{0}>0$ such that for all $t>t_{0}$ the map $\phi_{t}:\Omega\cup\mathcal{V}_{1}\to\mathbb{H}^{m+1}$ associated to $\rho_{t}=\rho+t$ is an embedded horospherically concave hypersurface.}
\end{quote}

Also, it is natural to relate analytic conditions along the boundary of a complete conformal metric with boundary (cf. Definition \ref{Def:CompletMetri}) and the boundary of the associated horospherically concave hypersurface, this is our next step. If we impose the condition that the horospherical metric $g=e^{2\rho}g_{0}$ has constant mean curvature along the boundary $\partial B_{r}(p)$, then we get information about the location of the boundary of the horospherically concave hypersurface, in fact, we will see that the boundary lies in an equidistant hypersurface (Proposition \ref{Ch3:Prop10}). Moreover, using the parallel flow, we obtain that the horospherically hypersurfaces is contained in one of the components in the Hyperbolic space $\h^{m+1}$ determined by such equidistant hypersurface where its boundary is contained (cf. Theorem \ref{Ch3:prop38}). Finally, we see how the parallel flow affects to elliptic problems for conformal metrics. Hence, we will be ready to prove our main results on conformal metrics by means of horospherically concave hypersurfaces in $\h^{m+1}$.

In what follows, we will deal with degenerate and non-degenerate elliptic problems for conformal metrics on either closed balls, punctured balls or compact annuli in the sphere $\S^{m}$.

As said above, J. Escobar proved in \cite{Esc} that, for the Yamabe problem with boundary in $\ov{\b^{m}}$, if the scalar curvature is zero then the mean curvature can not be negative. In Section \ref{SectBall}, we generalize such result to fully nonlinear degenerate  conformally invariant equations, that is,

\begin{quote}
{\bf Theorem \ref{Ch4:Theo.4.4}.} {\it Let $ (f , \Gamma )$ be an elliptic data for conformal metrics and let $c\leq 0$ be a constant. Then, there is no  conformal metric $g=e^{2\rho}g_{0}$ in $\overline{\mathbb{S}^{m}_{+} }$, where $\rho\in C^{2,\alpha}\left(\overline{ \mathbb{S}^{m}_{+} }\right)$, such that
    \begin{equation*}
      \left\{
               \begin{array}{ccccl}
                 f(\lambda(g))       &=& 0 & \text{ in } & \ov{ \mathbb{S}^{m}_{+} },          \\
                 h(g) &=& c & \text{ on } & \partial \mathbb{S}^{m}_{+},
               \end{array}
      \right.
    \end{equation*}where $\lambda(g)=(\lambda_{1},\ldots,\lambda_{m})$ is composed by the eigenvalues of the Schouten tensor of $g=e^{2\rho}g_{0}$.}
\end{quote}

The previous theorem can be extended to $m$-dimensional compact, simply-connected, locally conformally flat manifolds $(\man,g_{0})$ with umbilic boundary $\partial \man$ and $R(g_{0})\geq 0 $ on $\man$, using a result of F. M. Spiegel \cite{Spie},

\begin{quote}
{\bf Theorem \ref{Ch4:Theo.4.1.3}.} \textit{Set $(f,\Gamma)$ an elliptic data for conformal metrics and $c\leq 0$ a constant. Let $(\man,g_{0})$ be a $m$-dimensional compact, simply-connected, locally conformally flat manifold with umbilic boundary and $R(g_{0})\geq 0 $ on $\man$. Then, there is no conformal metric $g=e^{2\rho}g_{0}$, $\rho\in C^{2,\alpha}\left( \man \right)$, such that
  \begin{equation*}
      \left\{
               \begin{array}{ccccl}
                 f( \lambda(g) )       &=& 0 & \text{ in } &  \man ,          \\
                 h(g) &=& c & \text{ on } & \partial \man ,
               \end{array}
      \right.
    \end{equation*}where $\lambda(g)=(\lambda_{1},\ldots,\lambda_{m})$ is composed by the eigenvalues of the Schouten tensor of the metric $g=e^{2\rho}g_{0}$.}
\end{quote}

Also, using \cite{Spie}, we can extend the result of Cavalcante-Espinar \cite{CE} to locally conformally flat manifolds. Specifically,

\begin{quote}
\textbf{Theorem \ref{Ch4:Theo.4.1.4}.} \textit{Set $(f,\Gamma)$ an elliptic data for conformal metrics and $c\leq 0$ a constant. Let $(\man,g_{0})$ be a $m$-dimensional compact, simply-connected, locally conformally flat manifold with umbilic boundary and $R(g_{0})\geq 0 $ on $\man$. If there exists a conformal metric $g=e^{2\rho}g_{0}$, $\rho\in C^{2,\alpha}\left( \man \right)$, such that
  \begin{equation*}
      \left\{
               \begin{array}{ccccl}
                 f( \lambda(g) )       &=& 1 & \text{ in } &  \man ,          \\
                 h(g) &=& c & \text{ on } & \partial \man ,
               \end{array}
      \right.
    \end{equation*}where $\lambda(g)=(\lambda_{1},\ldots,\lambda_{m})$ is composed by the eigenvalues of the Schouten tensor of the metric $g=e^{2\rho}g_{0}$, then $\man$ is isometric to a geodesic ball in the standard sphere $\s ^{m}$.}
\end{quote}

In Section \ref{punct}, we study the degenerate case in the punctured geodesic ball with minimal boundary. Observe that, up to a conformal diffeomorphism acting on $\s^m$, we can consider the punctured geodesic ball as the punctured northern hemisphere. We obtain
\begin{quote}
\textbf{Theorem \ref{Ch4:Theo.4.4.11}.}\textit{Let $g=e^{2\rho}g_{0}$ be a conformal metric in $\ov{ \s^{m}_{+} }\setminus\{{\bf n}\}$ that is solution of the following degenerate elliptic problem:
\begin{equation*}
  \left\{
               \begin{array}{cccc}
                 f(\lambda(g))       &=& 0 & \quad\text{in } \ov{ \s^{m}_{+} }\setminus\{{\bf n}\},          \\
                 h(g) &=& 0 & \quad\text{on }\partial \s^{m}_{+},
               \end{array}
      \right.
  \end{equation*}then $g$ is rotationally invariant.}
\end{quote}

Now, if we assume that there is a solution $g=e^{2\rho}g_{0}$ in $\ov{\s^{m}_{+}}\setminus\{{\bf n}\}$ of
      \begin{equation*}
      \left\{
                   \begin{array}{cccc}
                     f(\lambda(g))       &=& 0 & \quad\text{in } \ov{ \S^{m}_{+} }\setminus \{{\bf n}\} ,          \\
                     h(g) &=& 0 & \quad\text{on }\partial \S^{m}_{+},
                   \end{array}
      \right.
     \end{equation*}
such that $\sigma=e^{-\rho}$ can be extended to a $C^{2}$ function $\tilde{\sigma}$ on $\ov{ \S^{m}_{+} }$ with $\tilde{\sigma}({\bf n})=0$. Such solution is called \textit{punctured solution}. In particular, Theorem \ref{Ch4:Theo.4.4.11} says that a punctured solution is rotationally symmetric. Finally, in the non-degenerate case in the punctured closed geodesic ball, we have
\begin{quote}
{\bf Theorem \ref{Ch4:Theo.4.4.12}.} \textit{Let $g=e^{2\rho}g_{0}$ be a conformal metric in $\ov{\s^{m}_{+}}\setminus\{{\bf n}\}$ that is solution of the following non-degenerate elliptic problem:
  \begin{equation*}
  \left\{
               \begin{array}{cccc}
                 f(\lambda(g))       &=& 1 & \quad\text{in } \ov{\s^{m}_{+}}\setminus\{{\bf n}\},          \\
                 h(g) &=& 0 & \quad\text{on }\partial \s^{m}_{+}, \\
               \end{array}
      \right.
  \end{equation*}then $g$ is rotationally invariant.}
\end{quote}

Next, in Section \ref{CompactAnnulus}, we deal with degenerate problems in the compact annulus $\ov{ \mathbb{A}(r) }$, $0<r<\pi/2$. We first observe that every solution to the degenerate problem in $\ov{ \mathbb{A}(r) }$ with minimal boundary is rotationally invariant. Also, if there is a solution of such problem then it is unique up to dilations.
\begin{quote}
  \textbf{Theorem \ref{Ch4:Theo.4.3.2}.} \textit{Set $r\in(0,\pi/2)$. If there is a solution $g=e^{2\rho}g_{0}$ of the following problem
  \begin{equation*}
  \left\{
               \begin{array}{cccc}
                 f(\lambda(g))       &=& 0 & \quad\text{in } \ov{ \mathbb{A}(r)},          \\
                 h(g) &=& 0 & \quad\text{on }\partial \mathbb{A}(r),
               \end{array}
      \right.
\end{equation*}then $g$ is rotationally invariant and unique up to dilations.}
\end{quote}

Moreover, under the existence of a punctured solution, we can prove:

\begin{quote}
{\bf Theorem \ref{Ch4:Theo.4.3.5}.} \textit{Set $r\in(0,\pi/2)$. If the degenerate elliptic data $(f,\Gamma)$ admits a punctured solution, then there is no solution of the following degenerate elliptic problem:
  \begin{equation*}
  \left\{
               \begin{array}{cccc}
                 f(\lambda(g))       &=& 0 & \quad\text{in } \ov{ \mathbb{A}(r) },          \\
                 h(g) &=& 0 & \quad\text{on }\partial \mathbb{A}(r).
               \end{array}
      \right.
\end{equation*}}
\end{quote}

In Section \ref{CompleteAnnulus}, we focus on different boundary conditions on the annulus. At one boundary component we will impose mild conditions on the metric and at the other we will impose constancy of the mean curvature of the conformal metric.

Our next result will say that any conformal metric $g=e^{2\rho}g_{0}$ in $$\mathbb{A}(r,\pi/2]:= \set{x \in \s ^m \, : \, \, r < d_{\s^m}(x,{\bf n}) \leq \pi /2}$$ satisfying certain property at its end and solution of a degenerate problem with non-negative constant mean curvature on its boundary, has unbounded Schouten tensor. In other words, we establish non-existence result for degenerate (and non-degenerate) elliptic equations in $\mathbb{A}(r,\pi/2]$. Specifically,

\begin{quote}
{\bf Theorem \ref{Ch4:Theo.4.3.6}.} {\it Let $r\in(0,\pi/2)$, $c\geq0$ be a non-negative constant and $g=e^{2\rho}g_{0}$ be a conformal metric in $\mathbb{ A }\left(r,\frac{\pi}{2}\right]$ that is solution of the following degenerate elliptic problem:
\begin{equation*}
  \left\{
               \begin{array}{cccc}
                 f(\lambda(g))       &=& 0 & \quad\text{in } \mathbb{A}(r,\pi/2],          \\
                 h(g) &=& c & \quad\text{on }\partial \s^{m}_{+}.
               \end{array}
      \right.
  \end{equation*}

If $e^{2\rho}+|\nabla\rho|^{2}:\mathbb{A}(r,\pi/2]\to\r$ is proper then $\lambda(g)$ is unbounded.}
\end{quote}

In the non-degenerate case, we have,

\begin{quote}
{\bf Theorem \ref{Ch4:Theo.4.3.7}.} {\it Let $0<r<\pi/2$, $c\in\r$ be a constant and $g=e^{2\rho}g_{0}$ be a conformal metric in $\mathbb{ A }\left(r,\frac{\pi}{2}\right]$ that is solution of the following non-degenerate elliptic problem:
\begin{equation*}
  \left\{
               \begin{array}{cccc}
                 f(\lambda(g))       &=& 1 & \quad\text{in } \mathbb{A}(r,\pi/2],          \\
                 h(g) &=& c & \quad\text{on }\partial \s^{m}_{+}, \\
                 \displaystyle{\lim_{x\to q}\rho(x)} & = & +\infty & \quad\forall\: q\in\partial B_{r}({\bf n}).
               \end{array}
      \right.
  \end{equation*}

Set $\sigma=e^{-\rho}$, if $|\nabla\sigma|^{2}$ is Lipschitz then $\nabla^{2} ( \sigma^{2} )$ is unbounded.}
\end{quote}

In the last Section \ref{2dim}, we see that we can extend the definition of the Schouten tensor to conformal metrics to the standard one in domains of the sphere $\s^{2}$. So, we can extend the notion of eigenvalues of the Schouten tensor and we can also speak of elliptic problems for conformal metrics in domains of the sphere $\s^{2}$. There, we observe that the Yamabe Problem reduces to the classical Liouville Problem. Hence, fully nonlinear equations for conformal metrics in domains on $\s^{2}$ can be regarded as a generalization of the Liouville Problem and we obtain analogous results in the two dimensional case. It is remarkable that there is a solution to the Yamabe Problem on the compact annulus with zero scalar curvature and minimal boundary, however, in dimension higher does not exist such solution. In fact, Theorem \ref{Ch4:Theo.4.3.5} is an extension of Escobar's Theorem (cf. \cite{Esc}), that is, if $m\geq 3$, then the Yamabe Problem on the compact annulus does not have solution with scalar curvature equals to zero and minimal boundary, however, this result can not be extended to $m=2$.

\bigskip
\section{Hypersurfaces via conformal metrics}\label{SectLocal}

In this section we first briefly sketch how a conformal metric defined in a subdomain of the sphere gives rise to an immersion into the Hyperbolic space and its geometric consequences (cf. \cite{BEQ,BEQ2,CE,Esp,EGM}). Second, we will extend some of the results contained in the works listed above to match our purposes.

Let $(\mathbb S ^m , g_0)$ be the standard $m-$sphere. Let $ \Omega\subset \mathbb S^n$  be a relatively compact domain and $g=e^{2\rho}g_0$ be a $C^\infty$ conformal metric on $ \Omega$. Assume that
$$
{\rm Sch}_g(p)<\frac{1}{2} \text{ for all } p\in  \Omega,
$$
that is, each eigenvalue of the Schouten tensor is less than $1/2$. Observe that we only need to assume that ${\rm Sch}_g < + \infty $ since we can always achieve this condition by a dilation $g_t=e^{2t}g$.

Denote by $\mathbb L ^{m+2}$ the standard Lorentz-Minkowski space, i.e,
$\mathbb L ^{m+2} = (\r^{m+2} , \metL{}{})$, where $\metL{}{}$ is the standard Lorentzian  metric given by
$$ \metL{}{} = - dx_0 ^2 + \sum _{i=1}^{m+1} dx_i^2 .$$

In this model one can consider
\begin{equation*}
\begin{array}{rcl}
\mathbb H ^{m+1} &=& \{  x \in \mathbb L ^{m+2} \, : \, \, \metL{x}{x}=-1 , \, x_0 >0\} ,\\
\mathbb S ^{m+1}_1 &=& \{  x \in \mathbb L ^{m+2} \, : \, \, \metL{x}{x}= 1\} , \\
\mathbb N ^{m+1}_+ &=& \{  x \in \mathbb L ^{m+2} \, : \, \, \metL{x}{x}= 0 , \, x_0 >0\} ,
\end{array}
\end{equation*}that is, the Hyperbolic Space, the deSitter Space and the positive Light Cone respectively.

Following \cite{BEQ,EGM}, one can construct a representation of $( \Omega , g)$ as an immersion
$\phi: \Omega \to \mathbb H^{m+1} \subset (\mathbb{L}^{m+2} , \meta{}{})$, endowed with a canonical orientation  $\eta :\Omega  \to \mathbb{S}^{m+1}_1 \subset \mathbb L ^{m+2}$, given by

\begin{equation}\label{phi}
\phi(x) = \frac{e^\rho}{2}\big(1+e^{-2\rho}(1+ |\nabla \rho  |^2)\big)(1,x)+e^{-\rho}(0,-x+\nabla\rho)
\end{equation} and whose Hyperbolic Gauss map is given by  $G(x)=x$.  In other words, one can construct a horospherically concave hypersurface $\Sigma=\phi(\Omega)$  with boundary $\partial\Sigma = \phi(\partial \Omega)$. Here, $| \, \cdot \, |$ and $ \nabla  \rho$ represent the norm and the gradient with respect to $g_0$.

Recall that the hyperbolic Gauss map is defined as follows. Let $x \in \Omega$ be a point in our domain and consider $p := \phi (x) \in \mathbb H ^{m+1}$ and $v := -\eta (x) \in T_p \mathbb H ^{n+1}$. Then, $G : \Omega \to \mathbb S ^m $ is defined by
$$ G( x ) := \lim _{t \to + \infty} \gamma _{p,v} (t)  \in \mathbb S ^m , $$where $\gamma : \mathbb R \to \mathbb H ^{m+1}$ is the complete geodesic parametrized by arc-length in $\mathbb H ^{m+1}$ passing through $p$ in the direction $v$.

\begin{remark}
{\it Note that, from \eqref{phi}, if $\rho \in C^{k+1} (\Omega)$ then the immersion $\phi $ is $C^k$, and also, the First and Second Fundamental Forms are $C^{k-1}$. It is worth noting that we only need $\rho \in C^2 (\Omega)$.}
\end{remark}

\medskip

We recall that  an immersion is horospherically concave if and only if the principal curvatures at any point are bigger than $-1$ with respect to the prescribed orientation $\eta $ (the inward orientation for a totally umbilical sphere). Also $ g :=  \metL{d\psi}{d\psi} = e^{2\rho} g_0 $ is a Riemmanian metric, being $\psi := \phi - \eta : \Omega \to \mathbb N ^{m+1}_+$ the light cone map, and it satisfies
$$\psi = e^{\rho}(1,x), \, x\in  \Omega ,$$that is, $g$ is nothing but the First Fundamental Form of $\psi$. Moreover, the principal curvatures, $\kappa _i $, of
$\Sigma $ and the eigenvalues, $\lambda _i$, of the Schouten tensor of $g$ are related by
\begin{equation}\label{lambdakappa}
\lambda _ i = \frac{1}{2} - \frac{1}{1+ \kappa _i} .
\end{equation}

\begin{remark}
The above representation formula \eqref{phi} can be seen as the hyperbolic analog to the representation formula for convex ovaloid in $\r ^{m+1}$ (cf. \cite{Fi1,Fi2})

Note that the above representation formula is local in nature, that is, the Hyperbolic Gauss map is (locally) a diffeomorphism (onto its local image) and, therefore, one can use \eqref{phi} to (locally) parametrize any horospherically concave ovaloid.
\end{remark}

Throughout this work we will use different models of the Hyperbolic space. We recall now the representation formula \eqref{phi} in the different models:

\begin{itemize}
\item Poincar\'{e} ball model: Given $x\in \Omega$, the representation formula \eqref{phi} is
\begin{equation}\label{Poincformul}
  \varphi_{P}(x)=\frac{1-e^{-2\rho(x)}+\left|\nabla e^{-\rho}(x) \right|^{2}}{ \left( 1+e^{-\rho(x)} \right)^{2}+\left|\nabla e^{-\rho}(x) \right|^{2}}x-\frac{1}{ \left( 1+e^{-\rho(x)} \right)^{2}+\left|\nabla e^{-\rho}(x) \right|^{2}}\nabla\left( e^{-2\rho}\right)(x).
\end{equation}

\item Klein model: Given $x\in \Omega$, the representation formula \eqref{phi} is
\begin{equation*}
  \varphi(x)=\frac{1-e^{-2\rho(x)}+\left|\nabla e^{-\rho}(x) \right|^{2} }{ 1+e^{-2\rho(x)} +\left|\nabla e^{-\rho}(x) \right|^{2} }x-\frac{1}{ 1+e^{-2\rho(x)} +\left|\nabla e^{-\rho}(x) \right|^{2}}\nabla\left( e^{-2\rho}\right)(x).
\end{equation*}

Also, set $\sigma=e^{-\rho}$, then the representation formula \eqref{phi} in the Klein model can be written as
\begin{equation}\label{Kleinformul}
  \varphi _K (x)=\frac{1-\sigma^{2}(x)+\left|\nabla \sigma(x)\right|^{2}}{ 1+\sigma^{2}(x) +\left|\nabla \sigma(x)\right|^{2}}x-\frac{1}{ 1+\sigma^{2}(x) +\left|\nabla \sigma(x)\right|^{2}}\nabla\left( \sigma^{2}\right)(x).
\end{equation}
\end{itemize}

\subsection{Properness}

Here, we study how the behavior of $\phi: \Omega\to\mathbb{H}^{m+1}$ depends on $\rho$. We characterize the properness of the associated horospherically concave hypersurface in terms of the behavior of $\rho$ along the boundary.

\begin{theorem}\label{Ch3:proper}
Given $\rho\in C^{1}\left( \Omega \right)$, the map $\phi:\Omega\to\mathbb{H}^{m+1}$ is proper if, and only if, $ \left|\rho \right|^2_{1,\infty}(x)\to\infty$ when $x\to p$, for every $p\in\partial\Omega$. Here,
$$ \left|\rho \right|^2_{1,\infty}(x) = \rho ^2 (x) + \abs{\nabla\rho (x)}^2.$$
\end{theorem}
\begin{proof}
From \eqref{Kleinformul}, the map associated to $\rho$, $\varphi _K:\Omega\to \left( \mathbb{B}^{m+1} , g_{K} \right)$, is given by
\begin{equation*}
  \varphi _K(x)=x-\frac{2\sigma^{2}(x)}{1+\sigma^{2}(x)+\left| \nabla \sigma(x)\right|^{2}}x
             -\frac{2\sigma(x)}{1+\sigma^{2}(x)+\left| \nabla \sigma(x)\right|^{2}}\nabla\sigma(x),
\end{equation*}
for all $x\in\Omega$, where $\sigma(x)=e^{-\rho(x)}$. Taking the Euclidean norm of $\varphi _K$ we obtain
\begin{equation*}
\left|\varphi _K(x)\right|^{2}=1-\left(\frac{2\sigma(x)}{1+\sigma(x)^{2}+|\nabla\sigma(x)|^{2}}\right)^{2} \text{ for every } x\in\Omega .
\end{equation*}

Hence, $\varphi _K$ is proper if, and only if,
\begin{equation*}
\lim_{x\to p}\left(\frac{1}{\sigma(x)}+\sigma(x)+\frac{|\nabla\sigma(x)|^{2}}{\sigma(x)}\right)=+\infty \text{ for all }  p\in\partial\Omega ,
\end{equation*}which is equivalent to
\begin{equation*}
\lim_{x\to p}\left(2\cosh\left(\rho(x)\right)+\frac{|\nabla \rho(x)|^{2}}{e^{\rho(x)}}\right)=+\infty  \text{ for all }  p\in\partial\Omega .
\end{equation*}

Finally, that is equivalent to
\begin{equation*}
\lim_{x\to p}\left[\rho(x)^{2}+|\nabla \rho(x)|^{2}\right]=+\infty  \text{ for all }  p\in\partial\Omega .
\end{equation*}
\end{proof}

In particular, we have
\begin{corollary}\label{Ch3:proper2}
If $\rho:\Omega\to\mathbb{R}$ is a proper smooth function then $\phi:\Omega\to\mathbb{H}^{m+1}$ is proper.
\end{corollary}

\begin{remark}
In \cite{BEQ}, the authors proved that if $\abs{\rho}_{1,\infty}^2 (x)$ diverges at the boundary then $\phi$ is proper.
\end{remark}

Also, as a consequence of the proof of the above theorem, we obtain another condition on $\rho$ that makes $\phi$ proper when $g=e^{2\rho}g_{0}$ is complete.

\begin{theorem}\label{Ch3:Theo3}
Let $g=e^{2\rho}g_{0}$ be a complete metric in $\Omega$ such that $\sigma=e^{-\rho}$ is the restriction of a continuous function defined in $\overline{\Omega}$. Then $\phi:\Omega\to\mathbb{H}^{m+1}$ is a proper map.
\end{theorem}
\begin{proof}
Since $g=e^{2\rho}g_{0}$ is a complete metric in $\Omega\subset \s ^{m}$, we have $\limsup_{x\to p} \rho(x)=+\infty$ for all $p\in\partial\Omega $, that is equivalent to $\liminf_{x\to p} \left[-\rho(x)\right]=-\infty$ for all $p\in\partial\Omega $. Let $H:\overline{\Omega}\to\mathbb{R}$ the continuous extension of $\sigma:\Omega\to\mathbb{R}$, then

\begin{equation*}
H(p)=\lim_{x\to p}\sigma(x)=\liminf_{x\to p} \sigma(x)=0 \text{ for all } p\in\partial\Omega .
\end{equation*}

Thus, $\lim_{x\to p}\rho(x)=+\infty$ for all $p\in\Omega $, which implies that
\begin{equation*}
\lim_{x\to p}\left[\rho(x)^{2}+|\nabla \rho(x)|^{2}\right]=+\infty \text{ for all }  p\in\partial\Omega ,
\end{equation*}that is, $\phi:\Omega\to\mathbb{H}^{m+1}$ is proper.
\end{proof}

\subsection{Invariance of the properness}\label{subsect3.1}

An interesting relation between conformal metrics and horospherically concave hypersurface is how they are related by dilations and geodesic flow. Let us explain this in more detail. We assume that $\phi:\Omega\to\mathbb{H}^{m+1}$ is an horospherically concave hypersurface in $\mathbb{H}^{m+1}$. When we move the hypersurface $\phi:\Omega\to\mathbb{H}^{m+1}$ using the unit normal vector field $-\eta$, we have a family of horospherically concave hypersurfaces $\left\{\phi_{t}:\Omega\to\mathbb{H}^{m+1}\, : \, \, t>0\right\}$. For every $t>0$,
\begin{equation*}
  \phi_{t}(x)=\cosh(t)\phi(x)-\sinh(t)\eta(x)\quad\text{for all }x\in\Omega ,
\end{equation*}i.e.,
\begin{equation*}
  \phi_{t}(x)=\frac{e^{t+\rho(x)}}{2}\left[1+e^{-2(t+\rho(x))} (1+\left| \nabla\rho(x) \right|^{2}) \right] \left( 1,x \right)
  +e^{-(t+\rho(x))}\left( 0,-x+\nabla\rho(x) \right) ,
\end{equation*}
for all $x\in\Omega$. Then the map $\phi_{t}:\Omega\to\mathbb{H}^{m+1}$ is well-defined with horospherical metric $g=e^{2t}g=e^{2(t+\rho)}g_{0}$.
That is, the map $\phi_{t}:\Omega\to\mathbb{H}^{m+1}$ is just obtained from the conformal metric $g_{t}=e^{2t}g$ by the representation formula. Since the eigenvalues of the Schouten tensor of $g_{t}$ are just the dilation by a factor of $e^{-2t}$ of the eigenvalues of the Schouten tensor of $g$, i.e., given $x\in\Omega$ and let $\lambda_{1},\ldots,\lambda_{m}$ be the eigenvalues of the Schouten tensor of $g$ at the point $x$, then the eigenvalues of the Schouten tensor of $g_{t}$ at the point $x\in\Omega$ are
\begin{equation}\label{Ch3:paraleig}
  \lambda_{i,t}=e^{-2t}\lambda_{i}\leq \lambda_{i}<\frac{1}{2}\quad\text{ for all }i=1,\ldots,m,
\end{equation}
then the map $\phi_{t}:\Omega\to\mathbb{H}^{m+1}$ is a horospherically concave hypersurface, and clearly its horospherical metric is $g_{t}=e^{2t}g=e^{2(t+\rho)}g_{0}$. In conclusion, if we take $t>0$, the conformal metric $g_{t}=e^{2t}g$ give rise to a horospherically concave hypersurface $\phi_{t}:\Omega\to\mathbb{H}^{m+1}$ with the natural orientation $\eta_{t}$ given by
\begin{equation}\label{Ch3:Eq:natorient}
  \eta_{t}(x)=\phi_{t}(x)-e^{t+\rho(x)}(1,x)\quad\text{for all }x\in\Omega.
\end{equation}

Then, one observation is the following:

\begin{proposition}\label{Ch3:Prop5}
Assume that $\phi:\Omega\to\mathbb{H}^{m+1}$ is proper, then $\phi_{t}:\Omega\to\mathbb{H}^{m+1}$ is also proper for every $t\in\mathbb{R}$.
\end{proposition}

\subsection{From immersed to embedded}

So far we have seen that the geodesic flow preserves the regularity of a horospherically concave hypersurface. Now, we study how an immersed horospherically concave hypersurface becomes embedded under the geodesic flow. We start by defining the meaning of a complete conformal metric with boundary in our situation.

\begin{definition}\label{Def:CompletMetri}
  Let $\Omega\subset \mathbb{S}^{m}$ be an open domain such that $\partial\Omega=\mathcal{V}_{1}\cup \mathcal{V}_{2}$ where $\mathcal{V}_{1}$ and $\mathcal{V}_{2}$ are disjoint compact submanifolds. We say that a conformal metric $g=e^{2\rho}g_{0}$, $\rho\in C^{2,\alpha}\left( \Omega\cup \mathcal{V}_{1} \right)$, is complete with boundary if given a divergent curve $\gamma:[0,1)\to\Omega$ then either
  \begin{itemize}
    \item $\displaystyle{\lim_{t\to1}\gamma(t)\in\mathcal{V}_{1}}$ and $\int_{0}^{1}\left| \gamma'(t) \right|_{g}dt<+\infty$, or
    \item $\displaystyle{\lim_{t\to1}\gamma(t)\in\mathcal{V}_{2}}$ and $\int_{0}^{1}\left| \gamma'(t) \right|_{g}dt=+\infty$.
  \end{itemize}
  In other words, $g$ is a complete metric on the manifold with boundary $\Omega\cup \mathcal{V}_{1}$.
\end{definition}

\begin{remark}
In the above definition, $\mathcal{V}_{2}$ can contain points, that is, it is permitted submanifolds that have dimension zero, or, even $\mathcal{V}_{2}$ could be empty.
\end{remark}

The next theorem shows that if we impose some extension condition on certain functions related to $\sigma=e^{-\rho}$ we can move along the geodesic flow and then we get an embedded hypersurface.

\begin{theorem}\label{Ch3:theo33}
Let $\rho\in C^{2,\alpha}\left(\Omega\cup \mathcal{V}_{1}\right)$ be such that $\sigma=e^{-\rho}\in C^{2,\alpha}\left( \Omega\cup \mathcal{V}_{1}\right)$ satisfies:
   \begin{enumerate}
     \item $\sigma\cdot\sigma$ can be extended to a $C^{1,1}$ function on $\overline{\Omega}$.
     \item $\left< \nabla\sigma,\nabla\sigma \right>$ can be extended to a Lipschitz function on $\overline{\Omega}$.
   \end{enumerate}

Then, there is $t_{0}>0$ such that for all $t>t_{0}$ the map $\phi_{t}:\Omega\cup\mathcal{V}_{1}\to\mathbb{H}^{m+1}$ associated to $\rho_{t}=\rho+t$ is an embedded horospherically concave hypersurface.
\end{theorem}
\begin{proof}
Let $\zeta:\mathbb{S}^{m}\to\mathbb{R}$ be a $C^{1,1}$-extension of $\sigma^{2}$ such that $\zeta>-\dfrac{1}{3}$, and $l\in C^{\infty}\left(\mathbb{S}^{m} \right)$ be a Lipschitz extension of $\left| \nabla\sigma \right|^{2}$ such that also $l>-\dfrac{1}{3}$. Then, for $t>0$, we have the following Lipschitz extension of $\varphi_{t}$ (we work on the Klein model):
  \begin{equation*}
    \Phi_{t}(x)=x-2\frac{e^{-2t}\zeta(x)}{ 1+e^{-2t}\left[\zeta(x)+l(x)\right] }x-\frac{e^{-2t}}{1+e^{-2t}\left[\zeta(x)+l(x)\right] }\nabla \zeta (x),\quad x\in\mathbb{S}^{m}.
  \end{equation*}

Since $\left\{\Phi_{t}\right\}_{t>0}$ converges to the inclusion $\mathbb{S}^{m}\hookrightarrow\mathbb{R}^{m+1}$ uniformly on $\s ^{m}$, there is $t_{0}>0$ such that for every $t>t_{0}$, the map $\Phi_{t}$ is embedded. Then there is $t_{0}>0$ such that for every $t>t_{0}$ the map $\varphi_{t}$ is embedded (cf. \cite{FukNak}).

Also, from the equation:
\begin{equation*}
  g^{-1}\text{Sch}(g)+\frac{1}{2}|\nabla\sigma|^2Id+\left<\nabla\sigma,\cdot \right>\nabla\sigma=\frac{1}{2}\sigma^{2}Id+\nabla^{2}\sigma^{2}\quad\text{in }\Omega
\end{equation*}
and the hypothesis we have that the eigenvalues of the Schouten tensor of $g=e^{2\rho}g_{0}$ are bounded in $\Omega$, so, we can choose $t_{0}>0$ large, such that for every $t>t_{0}$, the map $\phi_{t}:\Omega\to\h^{m+1}$ (in the Hyperboloid model) is a horospherically concave hypersurface (see \eqref{Ch3:paraleig}). This concludes the proof.
\end{proof}

Let us see the difference between the Klein and Poincar\'{e} models with a simple example. Let $\Sigma $ be the horospherically concave hypersurface associated to the function $e^{-\rho}=\sigma:\Omega=\{(x,y,z)\in\s^{2}\, : \,\,|z|<\cos(\pi/4)\}\to\mathbb{R}$ given by
\begin{equation*}
  \sigma(x,y,z)=\frac{2}{2+\sqrt{2}}\left( \sqrt{1-z^{2}}-\cos\left( \pi/4 \right)\right).
\end{equation*}

 \begin{figure}[!h]
   \subfloat[Poincar\'{e} model: The surface is transversal to the ideal boundary.]{\includegraphics[width=0.4\textwidth]{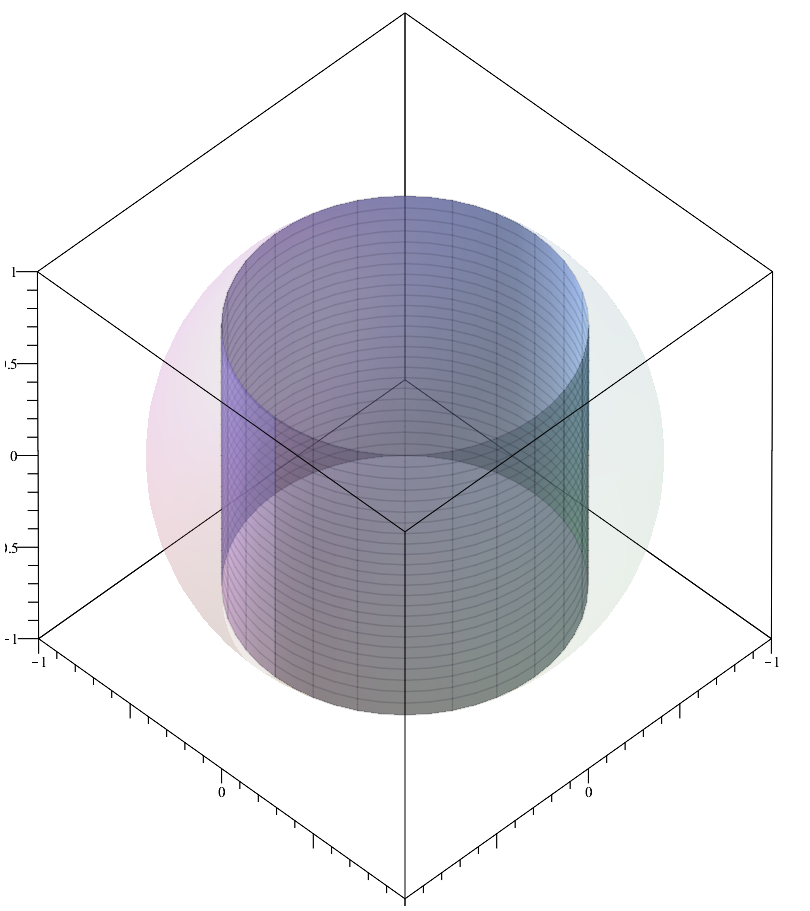}\label{f1}}
   \hfill
   \subfloat[Klein model: The surface is tangential to the ideal boundary.]{\includegraphics[width=0.4\textwidth]{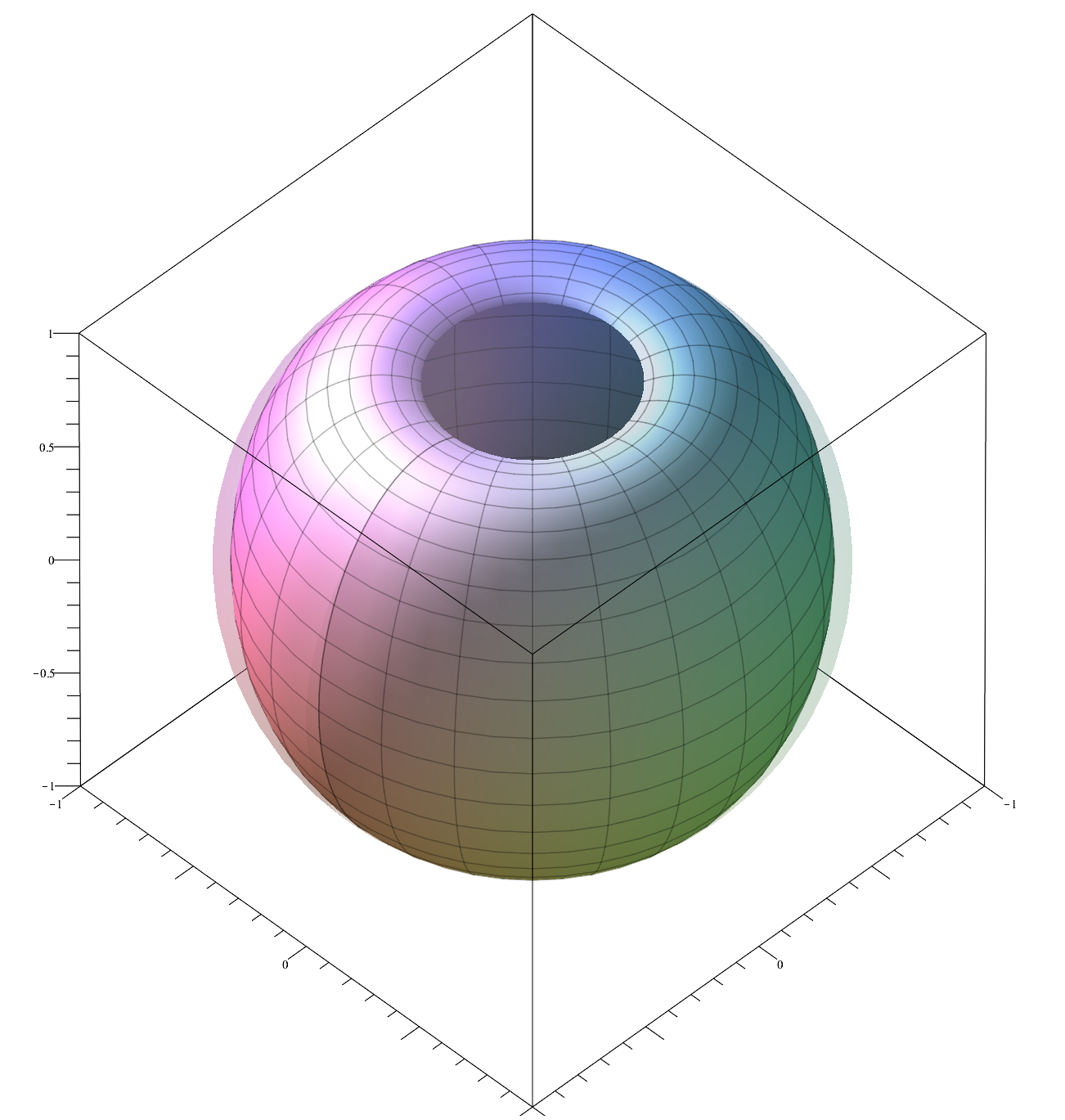}\label{f2}}
   \caption{Rotational surface}\label{fig0}
\end{figure}

One can easily observe (cf. Figure \ref{f1}), that in the Poincar\'{e} model $\Sigma$ is transversal to the ideal boundary. Nevertheless, in the Klein model, $\Sigma$ is tangential to the ideal boundary (cf. Figure \ref{f2}).

As we can see, the function $\sigma$ can be smoothly extended to $\s^{2}\setminus\{\pm e_{3}\} $, in fact, Figure \ref{f2} is a smooth extension of the surface in $\r^{3}$ in order to see the tangency of the surface with $\s^{2}$, the ideal boundary of $\h^{3}$.

\subsection{Conditions along the boundary}\label{Sec3.4}

Let $\Omega\subset\mathbb{S}^{m}$ be a domain such that $\partial\Omega=\mathcal{V}_{1}\cup\mathcal{V}_{2}$ and $\rho\in C^{2,\alpha}\left( \Omega \cup \mathcal{V}_{1}\right)$. Assume that $g:= e^{2\rho}g_0$  is complete in $\Omega \cup \mathcal{V}_{1}$ (see Definition \ref{Def:CompletMetri}). We study how conditions on either $\rho$ along $\mathcal{V}_{1}$, or geometric conditions on $\mathcal{V}_{1}$, influence the boundary of $\Sigma$.

In general, if $(\man,g_{0})$ is a Riemannian manifold with boundary $\partial \man$, $\nu$ is a unit normal vector field along $\partial \mathcal M$, $g=e^{2\rho}g_{0}$ is a conformal metric to $g_{0}$, $h_{0}$ the mean curvature of $\partial \man$ with respect to the metric $g_{0}$ and the unit normal vector field $\nu$, $h(g)$ the mean curvature of $\partial \man$ with respect to the metric $g$ and the normal vector field $\nu_{g}= e^{-\rho}\nu$, then
\begin{equation}\label{Ch3:equaMean}
   e^{\rho}\cdot h(g)+\frac{\partial \rho}{\partial \nu}=h_{0}  \text{ on }\partial \man,
 \end{equation}

Take $\sigma=e^{-\rho}$, then we have the following relation
\begin{equation}\label{Ch3:equaMean2}
   \frac{\partial \sigma}{\partial \nu}+h_{0}\cdot \sigma =h(g)  \text{ on } \partial \man,
 \end{equation}

If we consider the scaled metric $g_{t}=e^{2t}g$ in $\man$, where $t\in\r$, then the mean curvature of $\partial\man$ with respect to the unit normal vector field $\nu_{t}= e^{-t} \nu_{g}$ is
\begin{equation}\label{Ch3:Eq:meacurflow}
  h(g_{t})=e^{-t}h(g)\quad\text{on }\partial\man.
\end{equation}

Therefore, if $\partial\man$ is compact then $h(g_{t})$ goes to $0$ when $t$ goes to infinity. In our case $\man=\Omega\cup \mathcal{V}_{1}$ and $\partial \man=\mathcal{V}_{1}$.

Given $p\in\mathbb{S}^{m}$ and $r\in\left( 0,\dfrac{\pi}{2} \right]$, the geodesic ball of $\mathbb{S}^{m}$ centered at $p$ and radius $r$ is given by
  \begin{equation*}
    B_{r}(p)=\left\{ q\in \mathbb{S}^{m} \, : \, \, d_{\mathbb{S}^{m}}\left(q,p\right)<r \right\}
  \end{equation*}
  and its inward unit normal along $\partial B_{r}(p)$ is given by
  \begin{equation*}
     \nu(x)=\csc(r)p-\cot(r)x , \, x\in\partial B_{r}(p).
  \end{equation*}

We will see that any geodesic ball $B_{r}(p)$ has associated a unique totally geodesic hypersurface $E(a,0)\subset\mathbb{H}^{m+1}$, here we use the Hyperboloid model, such that
  \begin{equation*}
    \partial B_{r}(p)=\partial_{\infty}E(a,0),
  \end{equation*}
  where $a\in\mathbb{L}^{m+2}$ is a spacelike unit vector and $E(a,0)$ is defined by
  \begin{equation*}
    E(a,0)=\left\{y\in\mathbb{H}^{m+1} \, :\, \,  \ll y,a \gg=0 \right\}.
  \end{equation*}

We can explicitly get the vector $a$ from the center $p$ and the radius $r$, and vice-versa. Specifically, if $a = (a_0 ,\bar a)$, $\ll a,a\gg=1$, then
  \begin{equation*}
    p=\frac{1}{|\overline{a}|}\overline{a} \text{ and } \cot(r)=a_{0}, \,  r\in\left(0 , \frac{\pi}{2}\right).
  \end{equation*}

Now, we will study the boundary $\phi\left( \partial B_{r}(p) \right)$ of the associated horospherically hypersurface $\phi$ to $\rho$ when the boundary of $ B_{r}(p)$ has constant mean curvature with respect to the metric $g=e^{2\rho}g_{0}$. Consider a complete conformal metric $g=e^{2\rho}g_{0}$ in a domain $\Omega\cup \mathcal{V}_{1}\subset \overline{ B_{r}(p) }$ such that $\partial B_{r}(p)\subset \mathcal{V}_{1}$. Let $h(g)$ be the mean curvature of $\partial B_{r}(p)$ with respect to $g$ and the inward unit normal vector field $\nu_{g}= e^{-\rho}\nu$ along $\partial B_{r}(p)$, and $h_{0}=\cot(r)$.

Let $\phi: \Omega \cup \mathcal{V}_1\to\mathbb{H}^{m+1}$ be the associated horospherically concave hypersurface to the complete conformal metric $g$ in $\Omega\cup \mathcal{V}_1$. Then, a straightforward computation shows
  \begin{equation*}\label{Eq:BounInEqui}
    \ll\phi(x),h_{0}\left(1,x\right)+\left(0,\nu(x)\right)\gg=-h(g)  \text{ along }  \partial B_{r}(p),
  \end{equation*}where $\nu$ is the inward unit normal vector field along $\partial B_{r}(p)$ with respect to the standard metric $g_{0}$. Assume that $h(g)=c$, then
\begin{equation}\label{Ch3:MeanCte}
\ll\phi(x),h_{0}\left(1,x\right)+\left(0,\nu(x)\right)\gg=-c\text{ for all } x\in\partial B_{r}(p),
\end{equation}where
\begin{equation*}
\nu(x)=\csc(r)p-cot(r)x\text{ for all } x\in\partial B_{r}(p).
\end{equation*}

Set $a=h_{0}\left(1,x\right)+\left(0,\nu(x)\right)$, since $h_{0}=\cot(r)$, we have
  \begin{equation}\label{Eq:LikeSpace}
     a=\left(\cot(r),\frac{1}{\sin(r)}p\right) \text{ for all } x\in\partial B_{r}(p),
  \end{equation}i.e., $a$ only depends of $p$ and $r$.
  \begin{remark}
  In the particular case that $r=\pi/2$ and $p={\bf n}$, the north pole, we have
  \begin{equation*}\label{Eq:LikeSpace}
     a=\left(0,\ldots,0,1 \right)=\left( 0,e_{m+1} \right).
  \end{equation*}
  \end{remark}

  Then, from (\ref{Ch3:MeanCte}) and (\ref{Eq:LikeSpace}), we have that
\begin{equation*}
\phi\left( \partial \b_{r}(p)\right)\subset E(a,-c)=\left\{y\in\mathbb{H}^{m+1} \, : \, \, \ll y, a\gg=-c\right\}
  \end{equation*}which is an equidistant hypersurface to $E(a,0)$. In the case that $a=(0,\ldots,0,1)$, we just denote $E(-c)=E(a,-c)$. Summarizing, we have (see \cite[Claim E]{CE}):

\begin{proposition}\label{Ch3:Prop10}
Under the conditions above, assuming that $\mathcal{V}_{1}$ contains a component which is the boundary of a geodesic ball $\partial B_{r}(p)$, $p\in\mathbb{S}^{m}$, $r\in(0,\pi/2]$, and $h(g)=c$  along $\partial B_{r}(p)$,  then
\begin{equation*}
\phi\left(\partial B_{r}(p)\right) \subset E(a,-c),
\end{equation*}where $E(a,-c)$ is the totally geodesic hypersurface equidistant to $E(a,0)$ given by
\begin{equation*}
E(a,-c)=\left\{ y\in\mathbb{H}^{m+1} \, : \, \, \ll y,a \gg=-c \right\}
\end{equation*}
and $a=(\cot(r),\csc(r)p)$.
\end{proposition}

We can say even more, in fact, $\Sigma=\phi\left(\Omega\right)$ makes a constant angle with $E(a,-c)$ along $\phi\left( \partial B_{r}(p)\right)$. Without loss of generality and for simplicity, we will assume $\mathcal{V}_{1}=\partial B_{r}(p)$. A unit normal vector field along $E(a,-c)$ is given by $N(y)=\frac{1}{\sqrt{1+c^{2}}}\left(a-cy\right)$, for all $y\in E(a,-c)$. Since $\partial\Sigma\subset E(a,-c)$, we have the following result (see also \cite[Claim D]{CE}):

\begin{proposition}\label{proposition3.5}
Under the above conditions, it holds
\begin{equation*}
\ll N , \eta \gg=\frac{-c}{\sqrt{1+c^{2}}} \text{ along }\phi\left(\partial B_{r}(p)\right).
\end{equation*}

In other words, the angle $\alpha$ between $\Sigma$ and $E(a,-c)$ along $\phi\left(\partial B_{r}(p)\right)$ is constant and it satisfies
\begin{equation*}
\cos(\alpha)=-\frac{c}{\sqrt{1+c^{2}}}.
\end{equation*}
\end{proposition}

\subsection{Moving the hypersurface along the geodesic flow}

We will show the following boundary half-space property: Let $\Omega\subset\mathbb{S}^{m}_{+}$ be an open domain such that $\partial \Omega=\mathcal{V}_{1} \cup \mathcal{V}_{2}$, $\mathcal{V}_1 \cap \mathcal{V}_2 = \emptyset $, where the subset $\mathcal{V}_{1}$ is a compact hypersurface (not necessary connected) of $\mathbb{S}^{m}$ that contains $\partial\mathbb{S}^{m}_{+}$, and $\mathcal{V}_{2}$ is a finite union of disjoint compact submanifolds of $\mathbb{S}^{m}$, it might contain points. Let $\rho\in C^{2,\alpha} \left( \Omega \cup \mathcal{V}_{1} \right)$ be such that
\begin{equation*}
\lim_{x\to q}\left( e^{2\rho(x)}+\left|\nabla\rho(x)\right|^{2} \right)=+\infty \text{ for all } q\in\mathcal{V}_{2}.
\end{equation*}

Set
\begin{equation*}
  \mathcal{V}_{1}'=\mathcal{V}_{1}\setminus\partial\mathbb{S}^{m}_{+},
\end{equation*}we will show that, if $h(g)=c$ on $\partial \mathbb{S}^{m}_{+}$, then there exists $t_{1}\geq 0$ such that
\begin{equation*}
  \Sigma_{t}=\varphi_{t}\left(  \Omega\cup\mathcal{V}_{1}'\right) \subset C_{t}\text{ for all }t\geq t_{1} ,
\end{equation*}where $C_{t}\subset\mathbb{H}^{m+1}$ is the half-space determined by the equidistant $E(-e^{-t}c)$ that contains ${\bf n}$ at its boundary at infinity. Here, $\varphi _t$ stands for the immersion along the parallel flow in the Klein model (see \eqref{Kleinformul}). Specifically (see \cite[Claim C]{CE} in the compact case, i.e., for $\mathcal V _2 = \emptyset$):

\begin{theorem}\label{Ch3:prop38}
  Let $g=e^{2\rho}g_{0}$, $\rho\in C^{2,\alpha}\left( \Omega\cup\mathcal{V}_{1} \right)$, be a conformal metric in $ \Omega$ such that $\partial\S^{m}_{+}\subset\mathcal{V}_{1}$ and
  \begin{equation*}
    h(g)=c \text{ on } \partial\mathbb{S}^{m}_{+},
  \end{equation*}where $c\in\mathbb{R}$ is a constant. Assume that
  \begin{equation*}
    \lim_{x\to q}\left(e^{2\rho(x)}+\left| \nabla\rho(x) \right|^{2}\right)=+\infty \text{ for all } q\in \mathcal{V}_{2}.
  \end{equation*}

Then, there exists $t_{0}\geq 0$ such that for every $t>t_{0}$, the set $\phi_{t}\left( \Omega \cup \mathcal{V}_{1}' \right)$ (in the Klein model) is contained in the half-space determined by $E( -e^{-t}c)$ that contains ${\bf n}$ at its ideal boundary.
\end{theorem}
\begin{proof}
We work first in the Klein model. Set $K=\mathbb{S}^{m}_{+}\setminus\Omega$ and $int(K)=K\setminus\partial K$.
  For every $t>0$ we define the continuous extensions $\Phi_{t}: \overline{\mathbb{S}^{m}_{+}}\setminus int(K) \to\mathbb{R}^{m+1}$ of $\varphi_{t}:\overline{\mathbb{S}^{m}_{+}}\setminus K \to\mathbb{H}^{m+1}$ given by
  \begin{equation*}
    \Phi_{t}(x)=\left\{\begin{array}{ccc}
                          \varphi_{t}(x) & , &x\in \Omega\cup\mathcal{V}_{1}, \\
                          x & , &x\in \mathcal{V}_{2}.
                       \end{array}
                \right.
  \end{equation*}

Observe that $\left\{ \Phi_{t} \right\}_{t>0}$ converge to the inclusion $\overline{\mathbb{S}^{m}_{+}}\setminus int(K)\hookrightarrow\mathbb{R}^{m+1}$ when $t\to\infty$. We take an open set $V$ such that
  \begin{equation*}
    K\subset V \subset \overline{V}\subset \mathbb{S}^{m}_{+}.
  \end{equation*}

  Since $\overline{V}\setminus int(K)$ is compact, there exists $t_{1}>0$ such that, for all $t>t_{1}$, the set $\varphi_{t}\left( \overline{V}\setminus K \right)$ is in the half-space determined by the equidistant $E(-c)$ and contains $e_{m+1}$ at its ideal boundary.

  Now we consider the map $\varphi: \overline{\mathbb{S}^{m}_{+}}\setminus V \to\mathbb{H}^{m+1}$. We will prove that there is $t_{0}>t_{1}$ such that, for every $t>t_{0}$, the set $\varphi_{t}\left(\mathbb{S}^{m}_{+}\setminus V\right)$ is contained in the half-space determined by $E(-e^{-t}c)$ and contains ${\bf n}$ at its ideal boundary. That will finish the proof.

  Since $\varphi_{t}$ converges to the inclusion $\overline{\mathbb{S}^{m}_{+}}\setminus V\hookrightarrow\mathbb{R}^{m+1}$ uniformly and $\partial V$ is compact, there is $t_{1}>0$ such that, for all $t>t_{1}$, $\varphi_{t}\left( \partial V \right)$ is in the half-space determined by $E( -c)$ that contains $e_{m+1}$ at its ideal boundary.

From (\ref{Ch3:Eq:meacurflow}), the mean curvature of $\partial \s^{m}_{+}$ with respect to the scaled metric $g_{t}=e^{2t}g$, $t\in\r$, is $h(g_{t})=e^{-t}c$  along $\partial\s^{m}_{+}$.

Now, we pass to the Hyperboloid model. From Proposition \ref{Ch3:Prop10}, $\phi_{t}\left( \partial\s^{m}_{+} \right)\subset E(-e^{-t}c)$. Recall that $\phi _t$  stands for the immersion along the geodesic flow in the Hyperboloid model. We consider the following unit normal vector field along $E(-e^{-t}c)$
  \begin{equation*}
    N(y)=\dfrac{1}{\sqrt{1+s_{1}^{2}}}\left[ (0,{\bf n})-(e^{-t}c)\cdot y \right],\, y\in E(-e^{-t}c),
  \end{equation*} then the principal curvatures of the umbilic hypersurface $E(-e^{-t}c)$ with respect to $N$ are equal to $\frac{ce^{-t}}{ \sqrt{1+c^{2}e^{-2t}} }$.

Let $\kappa_{1,t},\ldots,\kappa _{m,t}$ be the principal curvatures of $\phi_{t}$. Then, for all $t>0$, we have
  \begin{equation*}
    \frac{1}{2}=e^{-2t}\lambda_{i}+\frac{1}{1+\kappa_{i,t}} \text{ on }\ov{\S^{m}_{+}}\setminus V , \text{ for all } i=1,\ldots,m,
  \end{equation*}
  where $\lambda_{1},\ldots,\lambda_{m}$ are the eigenvalues of the Schouten tensor of $g$. Since $\kappa _{i,t}$ goes to 1 uniformly on $\ov{ \S^{m}_{+} }\setminus V$ as $t$ goes to infinity, for $i=1,\ldots,m$, then there exists $t_{0}>t_{1}$ such that
  \begin{equation}\label{equation4.3}
     \kappa _{i,t}>\frac{1}{2}>-\frac{ce^{-t}}{ \sqrt{1+c^{2}e^{-2t}} } \text{ for all } t>t_{0} \text{ and for all } i=1,\ldots,m,
  \end{equation}on $\overline{\mathbb{S}^{m}_{+}}\setminus V$.

We claim that for every $t>t_{0}$, the set $\phi_{t}\left(\mathbb{S}^{m}_{+}\setminus V\right)$ is contained in the half-space determined by $E( -e^{-t}c)$  and contains ${\bf n}$ at its ideal boundary. If this were not the case, we consider the foliation by equidistant hypersurfaces $\{E(s)\}_{s\in\mathbb{R}}$ of the Hyperbolic space $\mathbb{H}^{m+1}$, given by
  \begin{equation*}
    E(s)=\left\{y\in\mathbb{H}^{m+1}\, : \, \,  \ll y,(0,e_{m+1}) \gg=s\right\} .
  \end{equation*}

\begin{figure}[!h]
   \subfloat[The equidistant $E(s)$ does not intersect the hypersurface.]{\includegraphics[width=0.45\textwidth]{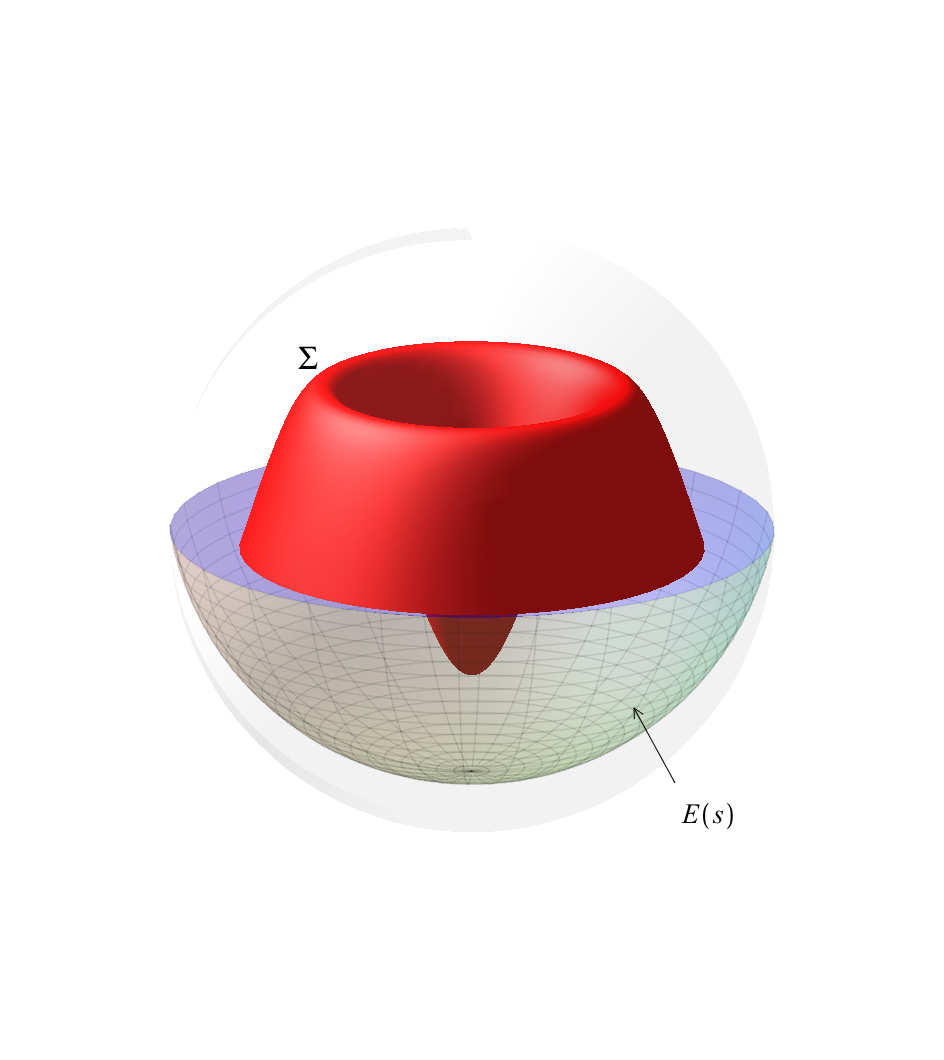}\label{grap1}}
   \hfill
    \subfloat[The equidistant $E(s_{1})$ intersects the hypersurface.]{\includegraphics[width=0.45\textwidth]{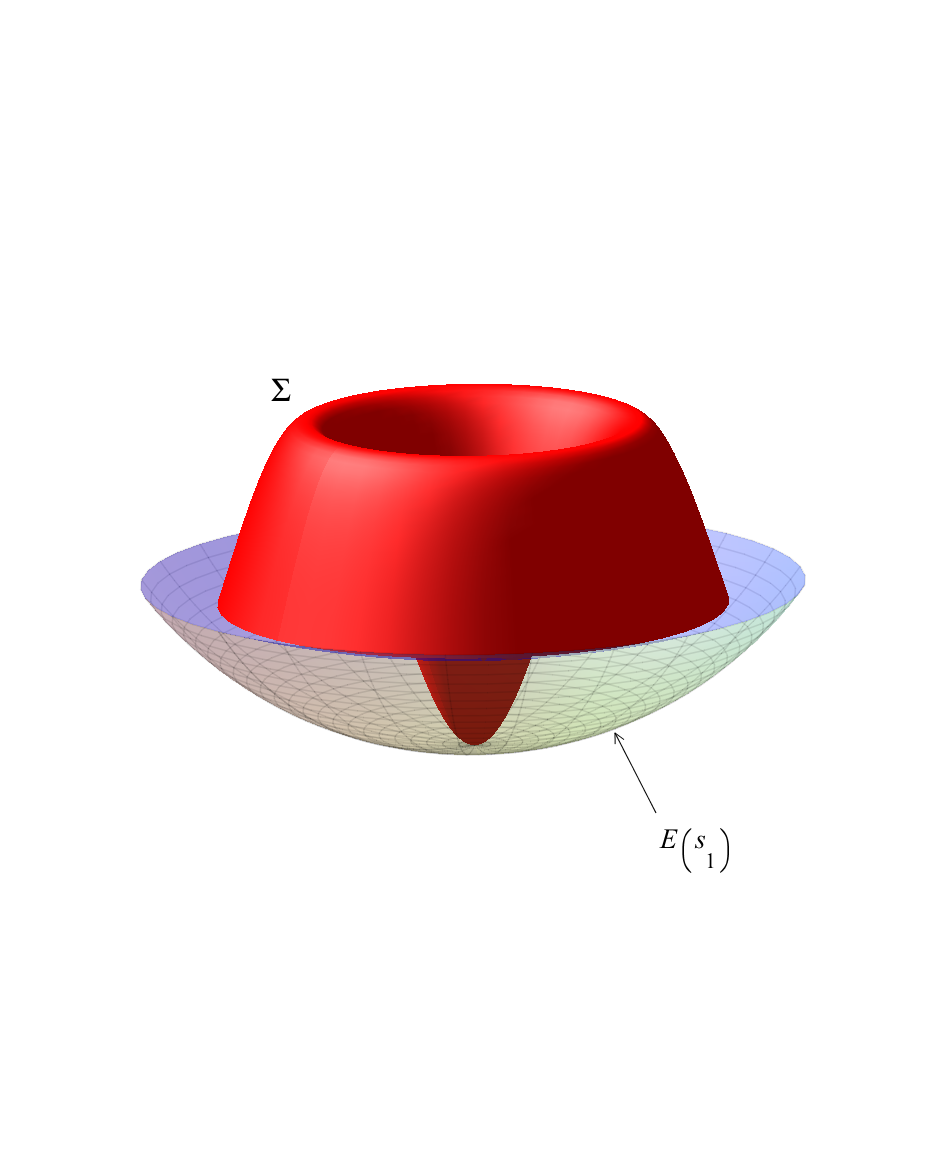}\label{grap2}}
   \hfill
   \subfloat[Opposite orientation of $\Sigma=Im(\phi_{t})$.]{ \includegraphics[width=0.45\textwidth]{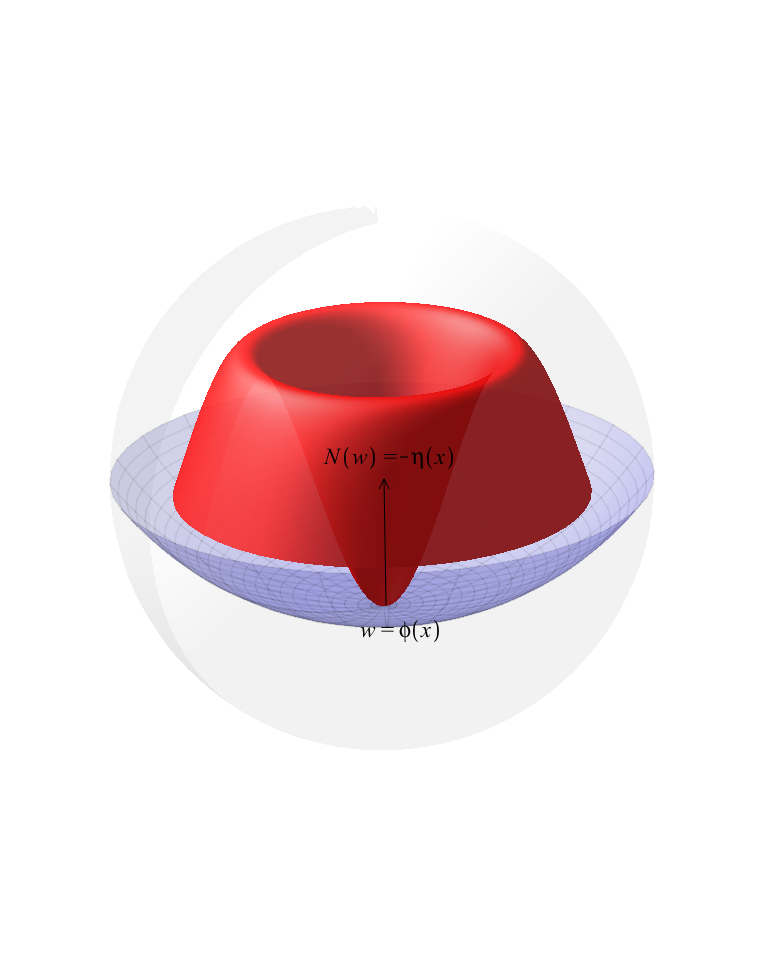}
  \label{grap3} }
\caption{Getting the first contact equidistant in the Poincar\'{e} ball model.}
\end{figure}

We consider the first equidistant hypersurface $E(s_{1})$ that intersects $\phi_{t}\left(\overline{\mathbb{S}^{m}_{+}}\setminus V\right)$ (cf. Figure \ref{grap1} and Figure \ref{grap2}), i.e.,
  \begin{equation*}
    E(s_{1})\cap \phi_{t}\left(\overline{\mathbb{S}^{m}_{+}}\setminus V\right)\neq\emptyset \text{ and } E(s)\cap\phi_{t}\left(\overline{\mathbb{S}^{m}_{+}}\setminus V\right)=\emptyset\text{ for all } s<s_{1}.
  \end{equation*}

Clearly $s_{1}\leq -c_{t}=-e^{-t}c$. We note that $E(s_{1})\cap \phi_{t}\left(\partial V\right)=\emptyset$ since $s_{1}\leq |c|$. Then
  \begin{equation*}
    E(s_{1})\cap \phi_{t}\left(\overline{\mathbb{S}^{m}_{+}}\setminus \overline{V}\right)\neq\emptyset.
  \end{equation*}

We claim that $E (s_{1})\cap \phi_{t}\left(\mathbb{S}^{m}_{+}\setminus\overline{V}\right) = \emptyset$, otherwise there would exist an interior contact point of $\phi_{t}\left(\mathbb{S}^{m}_{+}\setminus \overline{V}\right)$, say $x\in\mathbb{S}^{m}_{+}\setminus\overline{V}$ such that $w=\phi_{t}(x)\in E(s_{1})$.

Consider the normal vector field along $E(s_{1})$ that is defined in the Hyperboloid Model by $N(y)=\dfrac{1}{\sqrt{1+s_{1}^{2}}}\left[ (0,{\bf n})+s_{1}\cdot y \right]$, for all $y\in E(s_1)$.  The principal curvatures of $E(s_{1})$ with respect to this orientation are equal to
$ -\frac{s_{1}}{\sqrt{1+s_{1}^{2}}}$.

Along the horospherically concave hypersurface $\phi_{t}$, we consider the opposite  orientation to the canonical one, i.e, the unit normal vector field $\xi_{t}=-\eta_{t}$ (see (\ref{Ch3:Eq:natorient})). Then the principal curvatures are $\tilde{\kappa}_{i,t}=-\kappa_{i,t}$ with respect to that normal (cf. Figure \ref{grap3}).

Since the Hyperbolic Gauss map is the inclusion $\overline{\mathbb{S}^{m}_{+}}\setminus\overline{V}\hookrightarrow\mathbb{R}^{m+1}$, the normal vector field $\xi$ coincides with the normal $N$ along the equidistant $E(s)$ at the point $w$. The principal curvatures of $E(s_{1})$ with respect to $N$ at the point $w$ satisfy
  \begin{equation*}
    -\frac{s_{1}}{\sqrt{1+s_{1}^{2}}}\geq\frac{c_{t}}{\sqrt{1+c_{t}^{2}}} \text{ since } -s_{1}\geq c_{t}.
  \end{equation*}

Moreover, since $\phi_{t}\left( \mathbb{S}^{m}_{+}\setminus\overline{V}\right)$ is more convex than $E(s_{1})$ at the point $w$ we have
  \begin{equation*}
    \tilde{\kappa}_{i,t}\geq -\frac{s_{1}}{\sqrt{1+s_{1}^{2}}} \text{ at the point } w.
  \end{equation*}

That is,
  \begin{equation*}
    \kappa_{i,t}\leq \frac{s_{1}}{\sqrt{1+s_{1}^{2}}} \text{ at the point } w.
  \end{equation*}

So, at the point $w$, we have
  \begin{equation*}
    \kappa _{i,t}\leq -\frac{c_{t}}{\sqrt{1+c_{t}^{2}}} ,
  \end{equation*}but this contradicts (\ref{equation4.3}). Then, for every $t>t_{0}$, the set $\phi_{t}\left(\mathbb{S}^{m}_{+}\setminus V \right)$ is contained in the half-space determined by $E(-e^{-t}c)$ that contains ${\bf n}$ at its ideal boundary.

\end{proof}

\subsection{Elliptic problems for conformal metrics in domains of $\mathbb{S}^{m}$}\label{subsect1.6.1}

Consider $g=e^{2\rho}g_0$  a conformal metric, $\rho \in C^{2,\alpha}(\Omega)$, $\Omega\subseteq\s^{m}$, and denote by $\lambda (g)=(\lambda_1,\ldots, \lambda_m )$ the eigenvalues of the Schouten  tensor  of $g$.

We want to study partial differential equations for $\rho$ relating the eigenvalues of the Schouten tensor. For instance, the simplest example one may consider are the $\sigma_k$-Yamabe problems in $\Omega \subseteq \mathbb S^m$, that is, the  $k$-symmetric functions of the eigenvalues of the Schouten tensor. We are interested in the fully nonlinear case of this problem (see \cite{YLi06,YLi,LiLi1,LiLi2,LiLi3} and references therein).

Namely, given $(f,\Gamma)$  an elliptic data and a constant $c \geq 0$, find $\rho \in C^2 (\Omega)$ so that $g=e^{2\rho}g_0 $ is a solution of the problem
$$ f\left( \lambda(g)\right)= c  \text{ in } \Omega. $$

We must properly define the meaning of elliptic data $(f,\Gamma)$ for conformal metrics. Define
  \begin{equation*}
    \Gamma_{m}=\left\{x=(x_{1},\ldots,x_{m})\in\mathbb{R}^{m} \, : \,\, x_{i}>0, \, i=1,\ldots,m \right\}
  \end{equation*}
  and
  \begin{equation*}
    \Gamma_{1}=\left\{x=(x_{1},\ldots,x_{m})\in\mathbb{R}^{m}\, : \,\, \sum_{i=1}^{m}x_{i}>0\right\}.
  \end{equation*}

Let $\Gamma\subset\mathbb{R}^{m}$ be a convex open set satisfying:
\begin{enumerate}

\item[(C1)] It is symmetric. If $(x_{1},\ldots,x_{m})\in\Gamma$, then $(x_{i_{1}},\ldots,x_{i_{m}})\in\Gamma$, for every permutation $(i_{1},\ldots,i_{m})$ of $(1,\ldots,m)$.

\item[(C2)] It is  a cone. For every $t>0$, we have that $t(x_{1},\ldots,x_{m})\in\Gamma$ for  every $(x_{1},\ldots,x_{m})\in\Gamma$.

\item[(C3)] $\Gamma_{m}\subset \Gamma\subset\Gamma_{1}$.
\end{enumerate}

Then, we are ready to define:

\begin{definition}
We say that $(f ,\Gamma ) $ is an {\bf elliptic data} for conformal metrics if $\Gamma \subset \r ^{m}$ is a convex cone satisfying (C1), (C2) and (C3) and $ f\in C^{0}\left(\overline{\Gamma}\right)\cap C^{1}\left(\Gamma\right)$ is a function satisfying
    \begin{enumerate}
      \item $f$ is symmetric in $\Gamma$.
      \item $f|_{\partial\Gamma}=0.$
      \item $f|_{\Gamma}>0$.
      \item $f$ is homogeneous of degree 1.
      \item $\nabla f(x)\in\Gamma_{m}$ for every $x\in\Gamma$.
    \end{enumerate}
  \end{definition}

  This elliptic data is necessary for the definition of non-degenerate and degenerate elliptic problems.

\begin{definition}[\textbf{Elliptic problems for conformal metrics}]
Given an elliptic data $(f,\Gamma)$ for conformal metrics and $\Omega \subset \s ^m$ a domain:
  \begin{enumerate}
     \item The non-degenerate elliptic problem is to find a conformal metric $g=e^{2\rho}g_{0}$ such that
       \begin{equation*}
         f(\lambda(g)) = 1 \text{ in } \Omega
       \end{equation*}
       where $\lambda(g) = (\lambda_{1},\ldots,\lambda_{m})$ is composed by the eigenvalues of the Schouten tensor of $g$.
     \item The degenerate elliptic problem is to find a conformal metric $g=e^{2\rho}g_{0}$ such that
       \begin{equation*}
         f(\lambda(g)) = 0 \text{ in } \Omega
       \end{equation*}
       where $\lambda(g) = (\lambda_{1},\ldots,\lambda_{m})$ is composed by the eigenvalues of the Schouten tensor of $g$.
   \end{enumerate}
\end{definition}

There are two important remarks concerned to these equations. First, we have a particular solution in each case of special interest:
\begin{itemize}
\item {\it Non-degenerate elliptic problems:} a dilation of the standard metric is a solution to the problem, that is, there exists $\bar t \in \r $ so that $g : = e^{2 \bar t} g_0 $ is a solution of $f (\lambda (g)) =1 $ in $\Omega$.  This follows since $f$ is homogeneous of degree one.

\item {\it Degenerate elliptic problems:} A horosphere endowed with the inward orientation is horospherically concave as defined here. Then, the image of the Hyperbolic Gauss map is the whole sphere minus one point, the point at infinity where is based the horosphere. Hence, such horosphere defines a conformal metric $g := e^{2 \rho _H} g_0$ in $\s ^m \setminus \set{x}$, $x \in \s ^m$. Thus, for a domain $\Omega \subset \s ^m \setminus \set{x}$, the metric $g := e^{2 \rho _H} g_0$ is a solution of $f (\lambda (g)) =0 $ in $\Omega$. This follows since all the eigenvalues of the Schouten tensor associated to $g := e^{2 \rho _H} g_0$ are identically zero.
\end{itemize}

Second, it is about the Maximum Principle.

\begin{itemize}
\item  {\it Non-degenerate elliptic problems:} These equations satisfy the usual version of the Maximum Principle, Interior and Boundary (cf. \cite{LiLi1,LiLi2}).

\item {\it Degenerate elliptic problems:} These equations do not satisfy the usual Maximum Principle. Nevertheless, we will use a version developed by Y.Y. Li in \cite[Proposition 3.1]{YLi}.
\end{itemize}

\subsection{Dilation and elliptic problems for conformal metric}\label{sec3.6}

Given a conformal metric $g=e^{2\rho}g_{0}$, $\rho\in C^{\infty}\left( \Omega\right)$, that satisfies an elliptic problem for conformal metrics in $\Omega\subset\s^{m}$, i.e.,
\begin{equation*}
  f(\lambda(g))= cte \text{  in } \Omega,
\end{equation*}
where $f:\Gamma\to\mathbb{R}$ is an elliptic function for conformal metrics (see Subsection \ref{subsect1.6.1}), one can naturally ask the following question: given $t_{0}\in\mathbb{R}$, is the metric $g_{t_{0}}=e^{2t_{0}}g$ a solution of an elliptic problem for conformal metrics in $\Omega$? The answer is affirmative in the non-degenerate case and in the degenerate case. Let see in the non-degenerate case.

\begin{proposition}\label{Ch3:ParalProblem}
  Given a solution $g=e^{2\rho}g_{0}$, $\rho\in C^{\infty}\left(\Omega\right)$, of an elliptic problem with elliptic data $(f,\Gamma)$ and $t_{0}\in\mathbb{R}$, then $g_{t_{0}}=e^{2t_{0}}g$ is a solution of an elliptic problem that is given for the elliptic data $(f_{t_{0}},\Gamma)$ where
  \begin{equation*}
    f_{t_{0}}(x)=f\left( e^{2t_{0}}x\right)\quad\text{for all } x\in\Gamma.
  \end{equation*}
\end{proposition}
\begin{proof}
Since $\Gamma\subset\r^{m}$ is a cone then $e^{-2t_0}\Gamma=\Gamma$. Also $\partial \Gamma=e^{-2t_{0}} \partial\Gamma$, then
  \begin{equation*}
    f_{t_{0}}(x)=f(e^{2t_{0}}x)=0\quad\text{for all } x\in\partial\Gamma,
  \end{equation*}and
  \begin{equation*}
    \nabla f_{t_{0}}(x)=e^{2t_{0}}\nabla f(e^{2t_{0}}x)\in\Gamma_{m} \quad\text{for all } x\in\Gamma .
  \end{equation*}

It is clear that $f_{t_{0}}:\overline{\Gamma}\to\mathbb{R}$ is symmetric and
  \begin{equation*}
    f_{t_{0}}(\lambda(g_{t_{0}}))=f_{t_{0}}\left( e^{-2t_{0}}\lambda(g)\right)=f(e^{2t_{0}}e^{-2t_{0}}\lambda(g))=1 \text{ in }\Omega.
  \end{equation*}

Then, the conformal metric $g_{t_{0}}=e^{2t_{0}}g$ is a solution of the elliptic problem given by the elliptic data $(f_{t_{0}},\Gamma)$.
\end{proof}

\bigskip

\section{A non-existence Theorem on $\ov{ \mathbb{S}^{m} _+ }$}\label{SectBall}

As we have said at the Introduction, Escobar \cite{Esc2,Esc3} proved that there is no conformal metric $g=e^{2\rho}g_{Eucl}$ in the Euclidean unit ball $\b ^m$ with zero scalar curvature and nonpositive constant mean curvature along boundary, i.e., $h(g)\leq 0$ constant. We extend here such result for degenerate elliptic equations.

\begin{theorem}\label{Ch4:Theo.4.4}
Let $ (f , \Gamma )$ be an elliptic data for conformal metrics and let $c\leq 0$ be a constant. Then, there is no  conformal metric $g=e^{2\rho}g_{0}$ in $\overline{\mathbb{S}^{m}_{+} }$, where $\rho\in C^{2,\alpha}\left(\overline{ \mathbb{S}^{m}_{+} }\right)$, such that
    \begin{equation*}
      \left\{
               \begin{array}{ccccl}
                 f(\lambda(g))       &=& 0 & \text{ in } & \ov{ \mathbb{S}^{m}_{+} },          \\
                 h(g) &=& c & \text{ on } & \partial \mathbb{S}^{m}_{+},
               \end{array}
      \right.
    \end{equation*}where $\lambda(g)=(\lambda_{1},\ldots,\lambda_{m})$ is composed by the eigenvalues of the Schouten tensor of $g=e^{2\rho}g_{0}$.
\end{theorem}

\begin{proof}
The proof will be done by contradiction. Assume that there exists a conformal metric in $\overline{ \mathbb{S}^{m}_{+} }$, $g=e^{2\rho}g_{0}$, $\rho\in C^{2,\alpha}\left(\overline{ \mathbb{S}^{m}_{+} }\right)$, that solves the above problem.

Since $\overline{\mathbb{S}^{m}_{+}}$ is compact, up to a dilation of $g$, we can assume (see Subsection \ref{sec3.6}), without loss of generality, that all the eigenvalues of ${\rm Sch}(g)$ are less than $1/2$. By Theorems \ref{Ch3:theo33} and \ref{Ch3:prop38}, we can assume that the associated horospherically concave hypersurface $\phi:\overline{\mathbb{S}^{m}_{+}}\to\mathbb{H}^{m+1}$ is embedded and it is contained in the half-space determined by the equidistant hypersurface $E(-c)$ to the totally geodesic hypersurface $E(0)$, and such component contains ${\bf n}$ at its ideal boundary.
We notice that, in the Poincar\'{e} ball model, since $-c\geq0$, the equidistant $E(-c)$ is contained in the half-space determined by the totally geodesic hyperplane $E(0)$ that contains ${\bf n}$ at its ideal boundary.

We consider the horospherically concave hypesurface $\Sigma=\phi\left(\overline{\mathbb{S}^{m}_{+}}\right)$ in the Hyperbolic space $\mathbb{H}^{m+1}$ that is associated to $\rho$. Moreover, the angle between the hypersurface $\Sigma$ and the equidistant hypersurface $E(-c)$ along $\partial \Sigma$ is constant equals to
$$ \cos(\alpha) = -\dfrac{c}{\sqrt{1+c^{2}}},$$this follows from Proposition \ref{proposition3.5}.

Recall that the support function associated to horospheres  are solutions of the above degenerate elliptic problem. Now, we consider the foliation of the Hyperbolic space $\mathbb{H}^{m+1}$ by horospheres, $\{H(s)\}_{s\in\mathbb{R}}$, that have the same point at the ideal boundary of the Hyperbolic Space, $p_{\infty}={\bf s}$. We parametrize  by the signed distance, $s \in \r$, to the origin of the Poincar\'{e} ball model.

\begin{figure}[!ht]
   \subfloat[The hypersurface $\Sigma$ is in the convex side of the horosphere $H(s)$. ]{\includegraphics[width=0.4\textwidth]{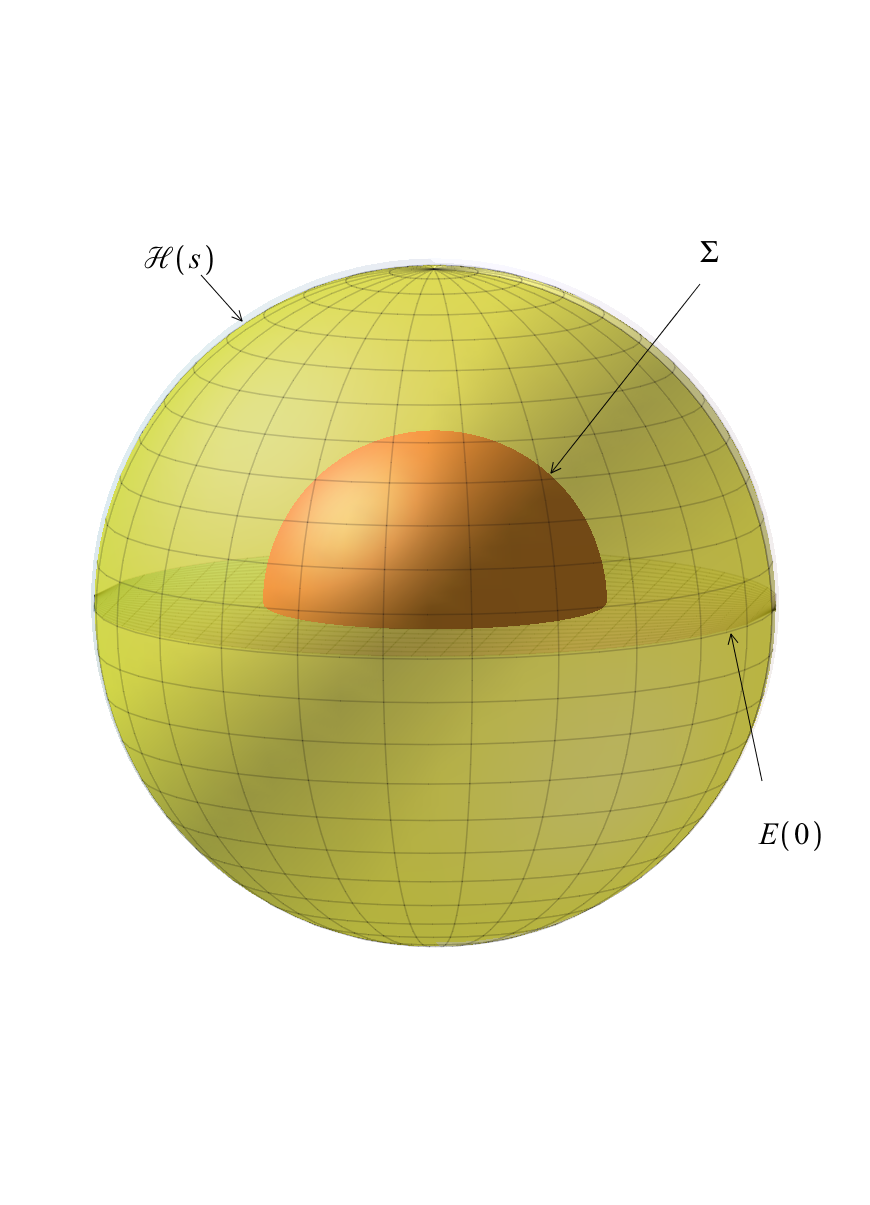}\label{grap410}}
   \hfill
   \subfloat[The horosphere $H(s)$ does not intersect $\Sigma$. ]{\includegraphics[width=0.4\textwidth]{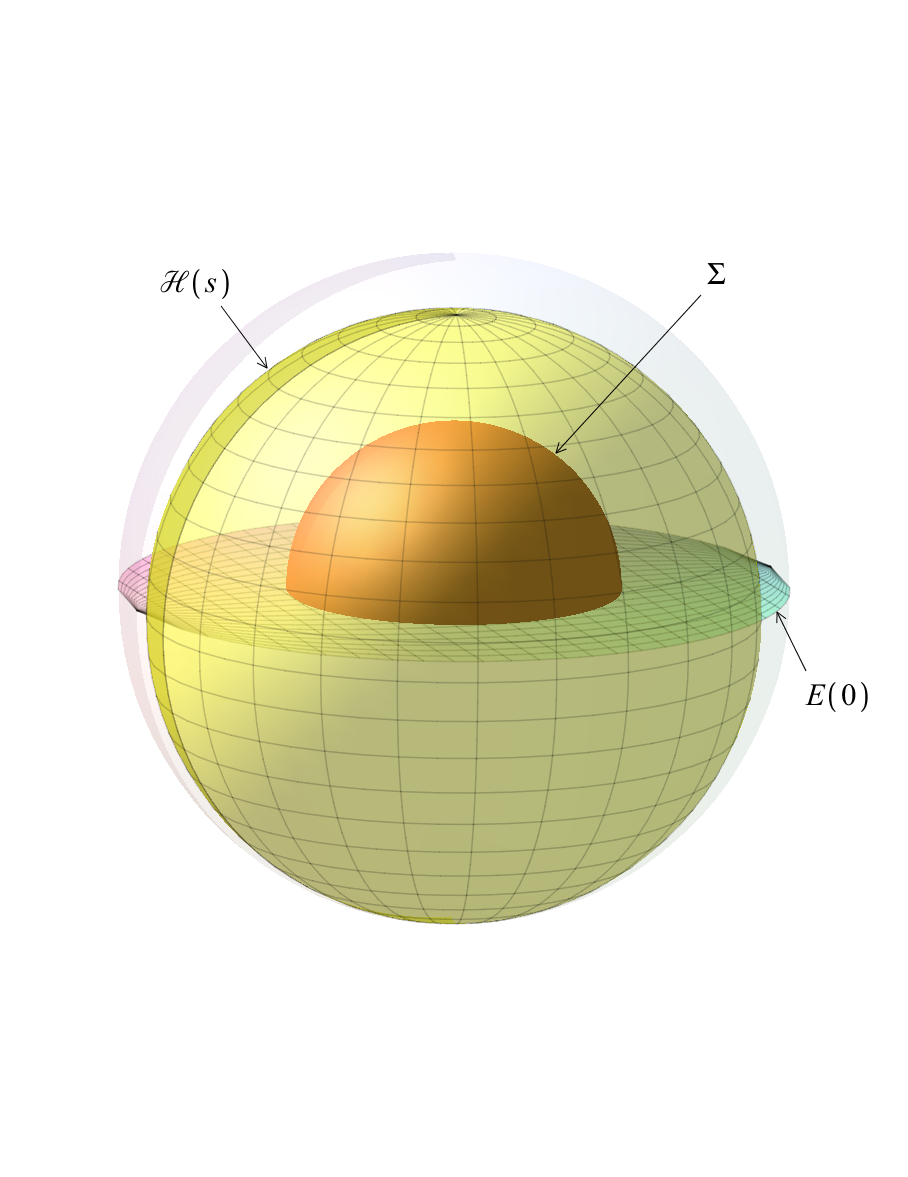}\label{grap411}}

   \subfloat[The horosphere $H(s)$ does not intersect the hypersurface $\Sigma$.]{\includegraphics[width=0.4\textwidth]{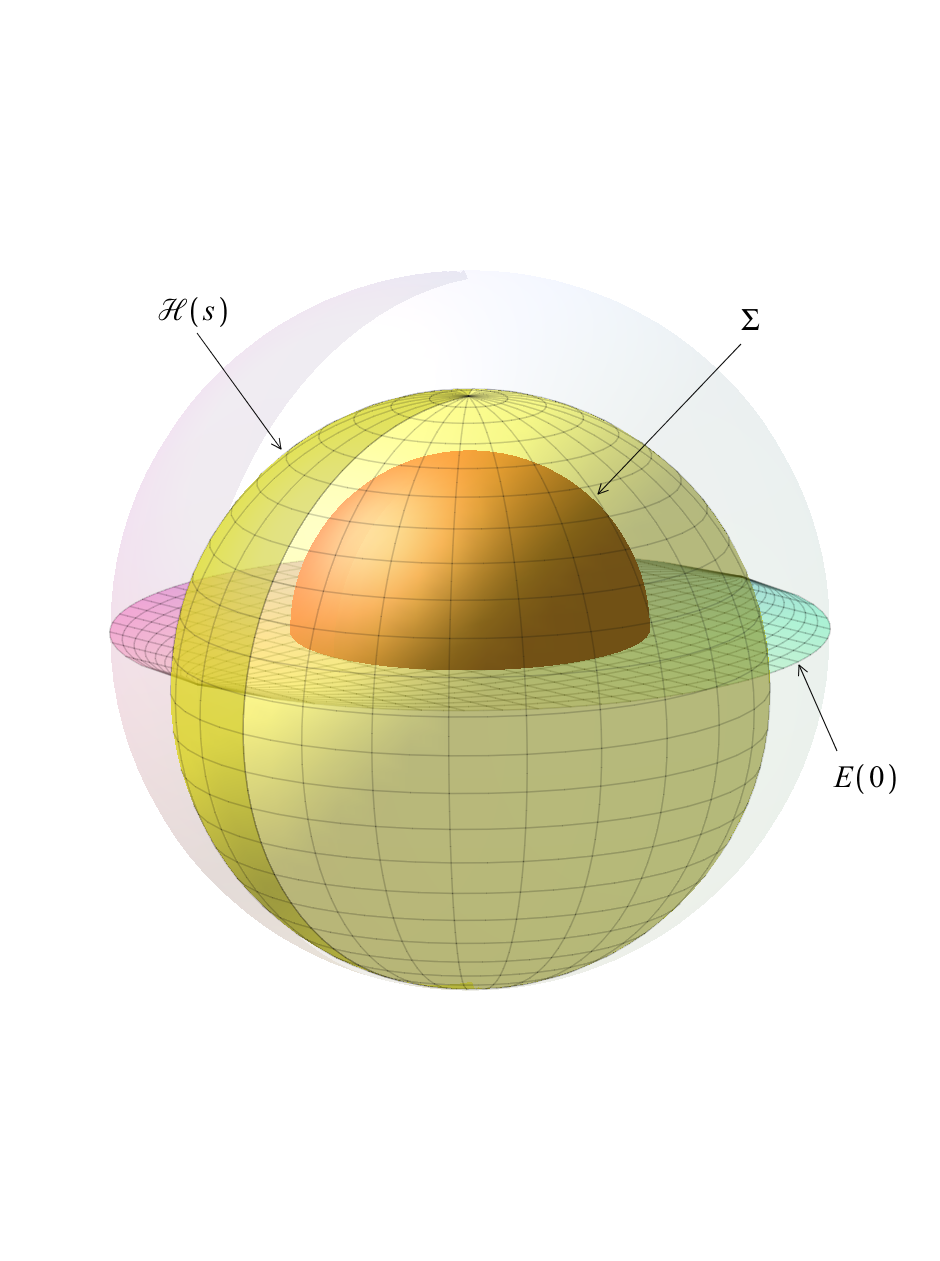}\label{grap412}}
   \hfill
   \subfloat[The horosphere $H(s_{1})$ touches the hypersurface $\Sigma$.]{\includegraphics[width=0.4\textwidth]{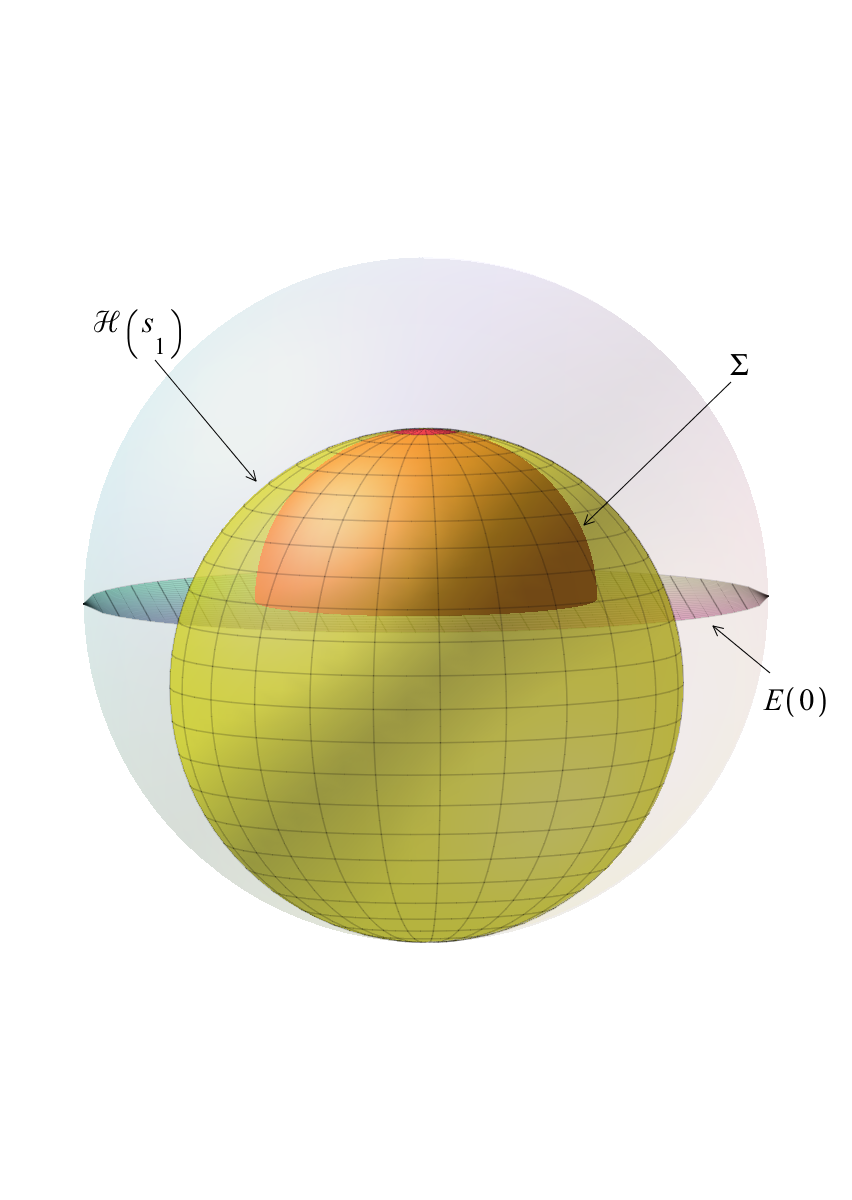}\label{grap413}}
   \caption{Getting the first contact horosphere in the Poincar\'{e} ball model. Case $c=0$.}\label{grap41}
\end{figure}

Since $\Sigma$ is compact, for $s$ large enough negatively, $\Sigma$ is completely contained in the mean-convex side of the horosphere $H(s)$ (cf. Figure \ref{grap410} when $c=0$). We continue increasing $s$ until we have the first contact point with $\Sigma$, let us say $s_{1}$ (cf. Figures \ref{grap411}, \ref{grap412} and \ref{grap413} when $c=0$). This means that, for every $s<s_{1}$, the hypersurface $\Sigma$ is in the interior of $H(s)$, $H(s)\cap \Sigma=\emptyset$, and $H(s_{1})\cap\Sigma$ is not empty.

Recall that
\begin{equation*}
  \cos(\alpha) = -\dfrac{c}{\sqrt{1+c^{2}}}\geq 0,
\end{equation*}since the angle $\alpha$ between the normal $\eta$ and the upward normal of $E(-c)$ is acute or $\pi/2$, we get that the contact point is at the interior of $\Sigma$.

The horosphere $H(s_{1})$ with its natural orientation has its own representation formula over $\s^{m}\setminus\{\bf s\}$.
We consider such formula restricted to $\ov{ \s^{m}_{+} }$ and we denote its support function by $\rho_{1}$. Since this horosphere is the first contact horosphere, we have that $\rho_{1}\geq\rho$ on $\ov{\s^{m}_{+}}$. Also, there is no boundary contact point of $\Sigma$, so there is $x\in\s^{m}_{+}$ such that $\rho_{1}(x)=\rho(x)$ and
\begin{equation*}
  \rho_{1}>\rho \text{ on }\partial\s^{m}_{+},
\end{equation*}
then, from \cite[Proposition 3.1]{YLi}, $\rho_{1}>\rho$ in $\ov{\s^{m}_{+}}$, which is a contradiction, since there is $x\in\s^{m}_{+}$ such that $\rho_{1}(x)=\rho(x)$.
\end{proof}

In the case that we consider a $m$-dimensional compact, simply-connected, locally conformally flat manifold $(\man,g_{0})$ with umbilic boundary and $R(g_{0})\geq 0 $ on $\man$, we have,

\begin{theorem}\label{Ch4:Theo.4.1.3}
Set $(f,\Gamma)$ an elliptic data for conformal metrics and $c\leq 0$ a constant. Let $(\man,g_{0})$ be a $m$-dimensional compact, simply-connected, locally conformally flat manifold with umbilic boundary and $R(g_{0})\geq 0 $ on $\man$. Then, there is no conformal metric $g=e^{2\rho}g_{0}$, $\rho\in C^{2,\alpha}\left( \man \right)$, such that
  \begin{equation*}
      \left\{
               \begin{array}{ccccl}
                 f( \lambda(g) )       &=& 0 & \text{ in } &  \man ,          \\
                 h(g) &=& c & \text{ on } & \partial \man ,
               \end{array}
      \right.
    \end{equation*}where $\lambda(g)=(\lambda_{1},\ldots,\lambda_{m})$ is composed by the eigenvalues of the Schouten tensor of the metric $g=e^{2\rho}g_{0}$.
\end{theorem}
\begin{proof}
By a result of F. M. Spiegel \cite{Spie}, there exists a conformal diffeomorphism between the Riemannian manifold $(\man,g_{0})$ and the standard closed hemisphere $\ov{ \S^{m}_{+} }$. This is done by using the developing map $\Phi$ of Schoen-Yau \cite{SchoenYau} and observing that umbilicity is preserved under conformal transformations. Hence, the boundary $\partial \man $ maps into a hypersphere in $\s^{m}$ by $\Phi$ and therefore, $\Phi$ maps $\man$ into the interior of a ball. Thus, the statement follows from Theorem \ref{Ch4:Theo.4.4}.
\end{proof}

Observe that the above argument can be used into the non-degenerate case studied by Cavalcante-Espinar \cite{CE}. Specifically,

\begin{theorem}\label{Ch4:Theo.4.1.4}
Set $(f,\Gamma)$ an elliptic data for conformal metrics and $c\leq 0$ a constant. Let $(\man,g_{0})$ be a $m$-dimensional compact, simply-connected, locally conformally flat manifold with umbilic boundary and $R(g_{0})\geq 0 $ on $\man$. If there exists a conformal metric $g=e^{2\rho}g_{0}$, $\rho\in C^{2,\alpha}\left( \man \right)$, such that
  \begin{equation*}
      \left\{
               \begin{array}{ccccl}
                 f( \lambda(g) )       &=& 1 & \text{ in } &  \man ,          \\
                 h(g) &=& c & \text{ on } & \partial \man ,
               \end{array}
      \right.
    \end{equation*}where $\lambda(g)=(\lambda_{1},\ldots,\lambda_{m})$ is composed by the eigenvalues of the Schouten tensor of the metric $g=e^{2\rho}g_{0}$, then $\man$ is isometric to a geodesic ball in the standard sphere $\s ^{m}$.
\end{theorem}
\begin{proof}
So, as above, there exists a conformal diffeomorphism from $\man $ to a ball into $\s ^m$ whose boundary has constant mean curvature. Moreover, the elliptic problem on $\man$ pass to an elliptic problem for the pushforward metric on the geodesic ball. Thus, by \cite[Theorem 1.1]{CE}, such pushforward metric has constant Schouten tensor, i.e., all the eigenvalues of the Schouten tensor are equal to the same constant. Hence, $\man$ is isometric to the geodesic ball in the sphere by the work of F. M. Spiegel \cite{Spie}.
\end{proof}


\section{Punctured geodesic ball}\label{punct}

Now, we see that any solution of a degenerate elliptic problem in the punctured geodesic ball $\ov{ \s^{m}_{+} }\setminus\{{\bf n}\}$ with minimal boundary is rotationally invariant.

\begin{theorem}\label{Ch4:Theo.4.4.11}
Let $g=e^{2\rho}g_{0}$ be a conformal metric in $\ov{ \s^{m}_{+} }\setminus\{{\bf n}\}$ that is solution of the following degenerate elliptic problem:
\begin{equation*}
  \left\{
               \begin{array}{cccc}
                 f(\lambda(g))       &=& 0 & \quad\text{in } \ov{ \s^{m}_{+} }\setminus\{{\bf n}\},          \\
                 h(g) &=& 0 & \quad\text{on }\partial \s^{m}_{+},
               \end{array}
      \right.
  \end{equation*}then $g$ is rotationally invariant.
\end{theorem}

\begin{proof}
Let us define $\tilde{\rho}:\s^{m}\setminus \{ \bf n,s \} \to\r$ as
\begin{equation*}
  \tilde{\rho}(x)=\left\{\begin{array}{cc}
                   \rho(x_{1},\ldots,x_{m}) & x\in\ov{\s^{m}_{+}}\setminus\{\bf n\}, \\
                   \rho(x_{1},\ldots,-x_{m}) & x\in\ov{\s^{m}_{-}}\setminus\{\bf s\} .
                 \end{array}
             \right.
\end{equation*}

First, we show that $\tilde{\rho}$ is $C^{1}$. Since
\begin{equation*}
  \frac{\partial \rho}{\partial x_{m+1}}=0 \text{ on }\partial \s^{m}_{+}
\end{equation*}
then $\tilde{\rho}\in C^{1}\left( \s^{m}\setminus\{\bf n,s\}\right)$. Then, the vector field $\nabla\tilde{\rho}:\s^{m}\setminus\{{\bf n,s}\}\to T\s^{m}$ given by
\begin{equation*}
  \nabla\tilde{\rho}(x)=\left\{\begin{array}{cc}
                   \nabla\rho(x) & x\in\ov{\s^{m}_{+}}\setminus\{\bf n\}, \\
                   R\nabla\rho (R(x)) & x\in\ov{\s^{m}_{-}}\setminus\{\bf s\},
                 \end{array}
             \right.
\end{equation*}
where $R:\r^{m+1}\to\r^{m+1}$ is the Euclidean reflection $R(x_{1},\ldots,x_{m})=(x_{1},\ldots,-x_{m})$, $(x_{1},\ldots,x_{m})\in\r^{m+1}$, is continuous.

Now, let us see that $\tilde{\rho}\in C^{2}$, that is, $\nabla\tilde\rho$ is $C^{1}$. Let $X_{1}=\nabla\rho|_{ \ov{\s^{m}_{+}}\setminus\{{\bf n}\} }$ and $X_{2}=\nabla\tilde{\rho}|_{ \ov{\s^{m}_{-}}\setminus\{ {\bf s} \} }$. Since $X_{1}=X_{2}$ on $\partial\s^{m}_{+}$, given $x\in\partial\s^{m}_{+}$ and $v\in T_{x}(\partial\s^{m}_{+})$ we have $\nabla_{v}X_{1}=\nabla_{v}X_{2}$.

Also, assume for a moment that $R(\nabla^{2}\rho(x)(e_{m+1}))=-\nabla^{2}\rho(x)(e_{m+1})$ for all $x\in\partial\S^{m}_{+}$. We have that for $x\in\partial\s^{m}_{+}$ and $v_{m}=e_{m+1}$,
\begin{equation*}
\begin{split}
  \nabla_{-v_{m}}X_{2} &=\frac{\partial X_{2}}{\partial (-v_{m})}(x)-<X_{2}(x),-v_{m}>x=\frac{\partial X_{2}}{\partial (-v_{m})}(x) \\
 &=R\nabla_{v_{m}}X_{1}=R \left( \nabla^{2}\rho(x)(e_{m+1}) \right) =-\nabla_{v_{m}}X_{1}.
\end{split}
\end{equation*}then $\nabla\tilde{\rho}$ is $C^{1}$, hence $\tilde{\rho}$ is $C^{2}$.

Since $\tilde{\rho}:\s^{m}\setminus\{{\bf n,s}\}\to\r$ is $C^{2}$ and it is a solution of a degenerate problem, from \cite{YLi}, $e^{2\tilde{\rho}}g_{0}$ is rotationally invariant, thus $g$ is rotationally invariant.

In order to conclude the proof, let us see that $R(\nabla^{2}\rho(x)(e_{m+1}))=-\nabla^{2}\rho(x)(e_{m+1})$ for all $x\in\partial\S^{m}_{+}$. For every $v\in T_{x}\partial\s^{m}_{+}$ we have that $<\nabla^{2}\rho(x)(v),e_{m+1}>=0$ then $\nabla^{2}\rho_{x}(V)\subset V$ where $V=T_{x}\partial\s^{m}_{+}$. Since $\nabla^{2}\rho(x):V\to V$ is symmetric, $\nabla^{2}\rho(x)(e_{m+1})\| e_{m+1}$. Thus $R(\nabla^{2}\rho(x)(e_{m+1}))=-\nabla^{2}\rho(x)(e_{m+1})$.
\end{proof}

\begin{definition}\label{Ch4:Def:Punct}
We say that a conformal metric $g=e^{2\rho}g_{0}$ in $\ov{ \S^{m}_{+} }\setminus\{{\bf n}\}$ is a \textit{punctured solution} of the problem
\begin{equation}\label{exequation}
  \left\{
               \begin{array}{cccc}
                 f(\lambda(g))       &=& 0 & \quad\text{in } \S^{m}_{+}\setminus \{{\bf n}\} ,          \\
                 h(g) &=& 0 & \quad\text{on }\partial \S^{m}_{+},
               \end{array}
      \right.
\end{equation}
if it is a solution of (\ref{exequation}) and $\sigma=e^{-\rho}$ admits a $C^2$-extension $\tilde{\sigma}:\ov{ \S^{m}_{+} } \to \mathbb{R}$ with $\tilde{\sigma}( {\bf n} )=0$.
\end{definition}

As a first observation, we have:
\begin{proposition}
Let $g=e^{2\rho}g_{0}$ be a punctured solution of (\ref{exequation}), then the associated map $\varphi_{P}:\ov{ \S^{m}_{+}}\setminus\{{\bf n }\} \to \r^{m+1}$ in the Poincar\'{e} ball model, i.e.,
   \begin{equation*}
      \varphi_{P}(x)=\frac{1-e^{-2\rho(x)}+\left|\nabla e^{-\rho(x)} \right|^{2}}{ \left( 1+e^{-\rho(x)} \right)^{2}+\left|\nabla e^{-\rho(x)} \right|^{2}}x-\frac{1}{ \left( 1+e^{-\rho(x)} \right)^{2}+\left|\nabla e^{-\rho(x)}\right|^{2}}\nabla\left( e^{-2\rho}\right)(x),
   \end{equation*}
for all $x\in\ov{ \S^{m}_{+}}\setminus\{{\bf n }\}$, admits a $C^{1}$- extension to $\ov{ \S^{m}_{+} }$.
\end{proposition}

  With this result, we can conclude:
  \begin{corollary}\label{Ch4:Punc}
     Punctured solutions are rotationally invariant.
  \end{corollary}

In the case of  a non-degenerate elliptic problem, we have the following

\begin{theorem}\label{Ch4:Theo.4.4.12}
Let $g=e^{2\rho}g_{0}$ be a conformal metric in $\ov{\s^{m}_{+}}\setminus\{{\bf n}\}$ that is solution of the following non-degenerate elliptic problem:
  \begin{equation*}
  \left\{
               \begin{array}{cccc}
                 f(\lambda(g))       &=& 1 & \quad\text{in } \ov{\s^{m}_{+}}\setminus\{{\bf n}\},          \\
                 h(g) &=& 0 & \quad\text{on }\partial \s^{m}_{+}, \\
               \end{array}
      \right.
  \end{equation*}then $g$ is rotationally invariant.
\end{theorem}
\begin{proof}
If $\rho$ admits a smooth extension $\tilde{\rho}:\ov{ \s^{m}_{+} }\to\r$ (just $C^2$) then the conformal metric $g_{1}=e^{2\tilde{\rho}}g_{0}$, defined in $\ov{ \s^{m}_{+} }$, is a solution of a non-degenerate problem with constant mean curvature on its boundary. Then $g_{1}$ is rotationally invariant, so, $g$ is rotationally invariant.

If $\rho$ does not admit smooth extension on $\ov{ \s^{m}_{+} }$, then by \cite{YLi06}, $\rho$ is rotationally invariant, then $g$ is rotationally invariant. This concludes the proof.
\end{proof}


\section{Compact Annulus}\label{CompactAnnulus}

We consider an annulus in $\s ^m$ whose boundary components are geodesic spheres, that is, the domains we will consider in this section are the sphere $\s ^m$ minus two geodesic balls.  Observe that, up to a conformal diffeomorphism, we can assume that the compact annulus is  $\ov{ \mathbb{A}(r) }$, $r\in(0,\pi/2)$, with minimal boundary. Here,
$$ \mathbb{A}(r) = \set{ x \in \s ^m \, : \, \, r < d_{\s^m}(x,{\bf n})< \pi /2 }. $$

We prove:

\begin{theorem}\label{Ch4:Theo.4.3.2}
Set $r\in(0,\pi/2)$. If there is a solution $g=e^{2\rho}g_{0}$ of the following problem
  \begin{equation*}
  \left\{
               \begin{array}{cccc}
                 f(\lambda(g))       &=& 0 & \quad\text{in } \ov{ \mathbb{A}(r)},          \\
                 h(g) &=& 0 & \quad\text{on }\partial \mathbb{A}(r),
               \end{array}
      \right.
\end{equation*}then $g$ is rotationally invariant and unique up to dilations.
\end{theorem}
\begin{proof}
Let $g_{1}$ be a solution of the above degenerate problem. First, we show that $g_{1}$ is rotationally invariant. Using the same techniques than  Cavalcante-Espinar in \cite{CE}, we get a conformal metric in $\s^{m}\setminus\{\bf n,s\}$ that is solution of the degenerate problem. Let us sketch this for the reader convenience. Using the representation formula, the horospherical hypersurface associated to $g_{1}$ is contained in the slab determined by two parallel hyperplanes and, by the boundary condition, the boundaries meet orthogonally such hyperplanes (cf. Figure \ref{gr4c1}). Thus, we can reflect this annulus with respect to the hyperplanes and we obtain an embedded complete horospherically concave hypersurface $\tilde{\Sigma_{1}}$ whose boundary at infinity are the north and south pole and, by construction, the extension $\tilde{\Sigma_{1}}$ of $\Sigma_{1}$ is contained in the interior of an equidistant hypersurface to the geodesic joining the north and south pole, the radius of such equidistant is determined by the original annulus, then its horospherical metric $\tilde{g_{1}}$ is a solution to the degenerate problem in $\s^{m}\setminus\{\bf n,s\}$ (cf. Figure \ref{gr4c2}). In fact, the metric $\tilde{g_{1}}$ is $C^{2}$ in $\s^{m}\setminus\{\bf n,s\}$ (one can proceed analogously to the proof given in Theorem \ref{Ch4:Theo.4.4.11}). From \cite{YLi}, we have that $\tilde{g_{1}}$ is rotationally invariant, thus $g_{1}$ is rotationally invariant.

  \begin{figure}[!tbp]
   \subfloat[The hypersurface $\Sigma$ is between two totally geodesic hypersurfaces: $E(a,0)$ and $E(0)$. ]{\includegraphics[width=0.4\textwidth]{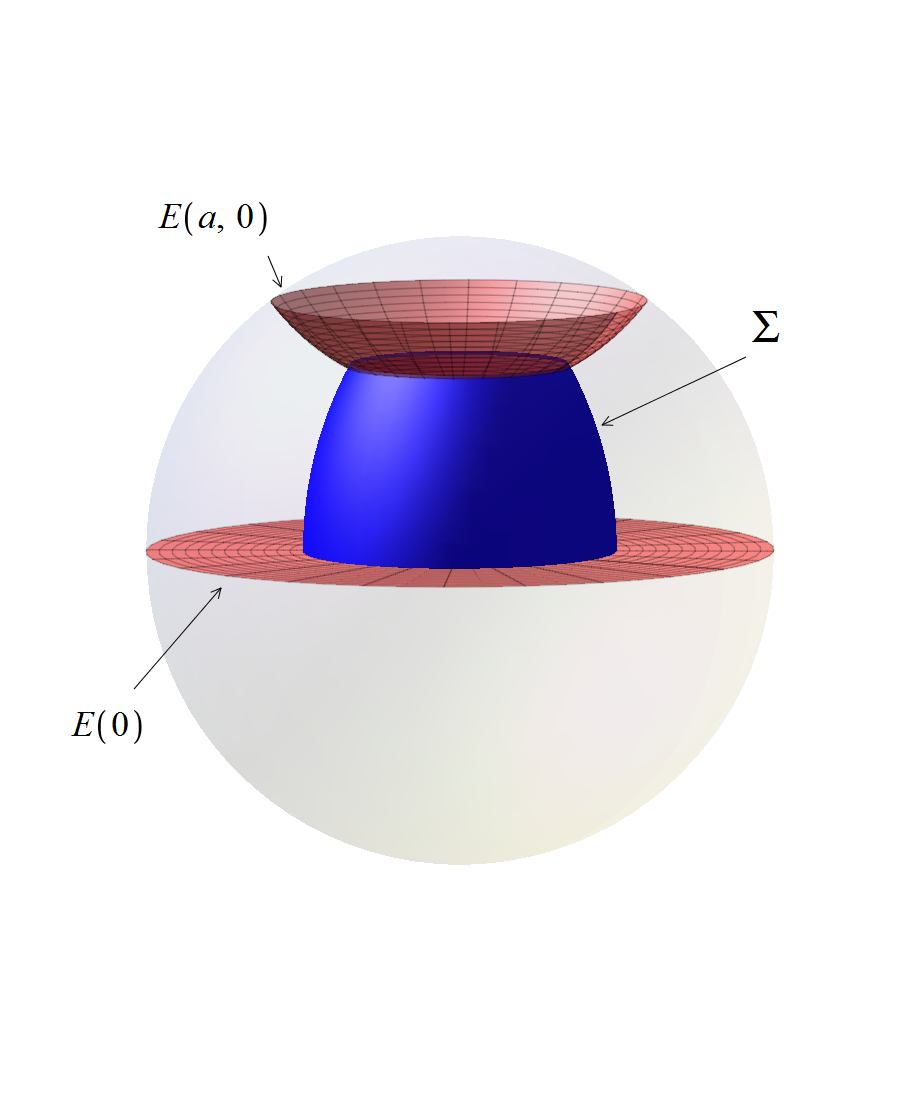}\label{gr4c1}}
   \hfill
   \subfloat[Extension $\tilde{\Sigma}$ of $\Sigma$ by reflection w.r.t. totally geodesic hypersurfaces. ]{\includegraphics[width=0.4\textwidth]{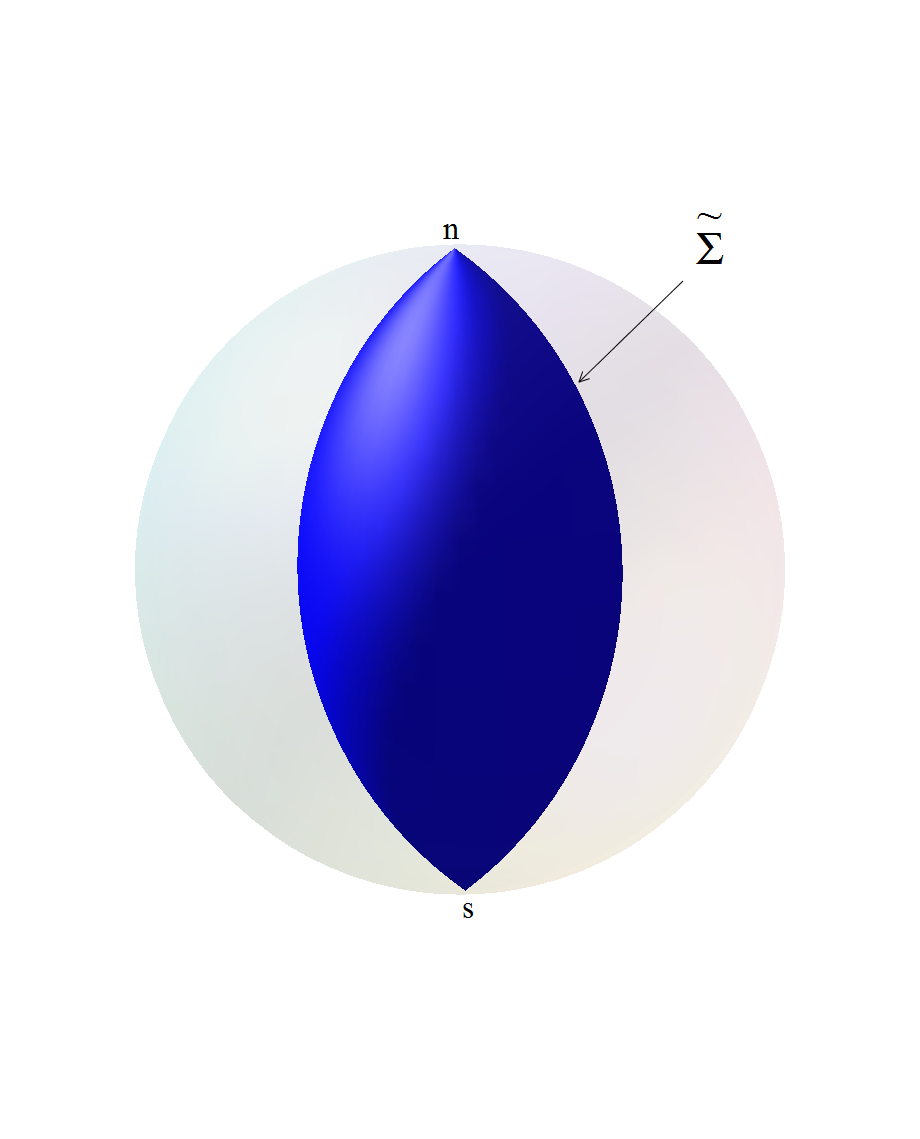}\label{gr4c2}}
   \caption{Extension of the compact hypersurface $\Sigma$.}\label{grap4c}
  \end{figure}

  Now, we show that it is unique up to dilations. Again, using the parallel flow, we can assume that the hypersurface $\Sigma _1$ associated to $g_{1}$ is an embedded horospherically concave hypersurface.

By Theorems \ref{Ch3:theo33} and \ref{Ch3:prop38}, we can suppose that the hypersurface is between the totally geodesic hypersurfaces $E(0)$ and $E(a,0)$, where $a=\left( \cot(r),\frac{1}{\sin(r)}{\bf n}\right)$. Also, the hypersurfaces along the parallel flow will remain in such slab as embedded horospherically concave hypersurfaces.

Let $\Sigma$ be the hypersurface associated to $g$ and $\Sigma _1^t$ the parallel flow of $\Sigma _1$. Since $\Sigma$ and $\Sigma_{1}$ are compact, there exists $t_0 \geq 0$ so that $\Sigma_{1}^{t_0}$ has no intersection with $\Sigma$. That is, in the Poincar\'{e} ball model,  $\Sigma_{1}^{t_0}$ is in the exterior of $\Sigma$, thus we can decrease $t$ until $\Sigma$ touches $\Sigma_{1}^{t}$ for the first time. We denote this first contact hypersurface as $\Sigma _1$. We have a contact point.

  We claim that there is a boundary contact point. If not, let $\rho_{1}$ be the support function of $\Sigma_{1}$ and $\rho$ be the support function of $\Sigma$. We have that there is $x\in \mathbb{A}(r)$ such that $\rho_{1}(x)=\rho(x)$, and, also,
  \begin{equation*}
    \rho_{1}>\rho \text{ on }\partial\mathbb{A}(r),
  \end{equation*}
  by \cite[Proposition 3.1]{YLi}, $\rho_{1}(x)>\rho(x)$, which is a contradiction.

  Even more, we claim that both components of $\partial\Sigma_{1} = \partial _1 \cup \partial _2 $ have a contact point with $\partial\Sigma = \tilde \partial _1 \cup \tilde \partial _2$, that is, $\partial _i \cap \tilde \partial _i \neq \emptyset$, $i=1,2$. Suppose that one component of $\partial\Sigma_{1}$ does not have a contact point with $\partial\Sigma$, we reflect both hypersurfaces with respect to the hyperplane that contains the component of $\partial\Sigma_{1}$ that touches $\partial\Sigma$. We have an extension $\Sigma_{1}'$ of $\Sigma_{1}$ and an extension $\Sigma'$ of $\Sigma$ such that they have an interior contact point and there is no boundary contact point. Then, repeating the arguments of the above paragraph, we get a contradiction.

Similar arguments show that at each level of $\mathbb{A}(r)$ there is a contact point between $\Sigma_{1}$ and $\Sigma$, i.e., for every $s\in(r,\pi/2)$ there is $x\in\mathbb{A}(r)$, with $d_{\s^{m}}(x,{\bf n})=s$, such that the support function of $\Sigma_{1}$ and the function support of $\Sigma$ take the same value at $x$. The proof is also by contradiction. If there is no such $x$, then we use the reflection to attain a contradiction as above.

  We know that $\rho_{1}$ and $\rho$ are rotationally invariant. Since at every level of $\mathbb{A}(r)$, there is a point at that level such that $\rho_{1}$ and $\rho$ are equal, we have that $\rho_{1}=\rho$ on that level, then $\rho_{1}=\rho$ on $\ov{ \mathbb{A}(r) }$. So, they are equal up to parallel flow. That is, $g_{1}$ is a dilation of $g$. This concludes the proof.
\end{proof}

In case that the above elliptic data admits a punctured solution then there is no solution of the above problem. Specifically:

\begin{theorem}\label{Ch4:Theo.4.3.5}
Set $r\in(0,\pi/2)$. If the degenerate elliptic data $(f,\Gamma)$ admits a punctured solution, then there is no solution of the following degenerate elliptic problem:
  \begin{equation*}
  \left\{
               \begin{array}{cccc}
                 f(\lambda(g))       &=& 0 & \quad\text{in } \ov{ \mathbb{A}(r) },          \\
                 h(g) &=& 0 & \quad\text{on }\partial \mathbb{A}(r).
               \end{array}
      \right.
\end{equation*}
\end{theorem}

\begin{proof}
We will prove this by contradiction. Suppose that there is a solution of the above problem. Arguing as in Theorem \ref{Ch4:Theo.4.3.2}, such solution can be extended to a solution $g=e^{2\rho}g_{0}$ on $\ov{ \S^{m}_{+} } \setminus \{ {\bf n} \}$. Let $\Sigma$ be the associated hypersurface to $g$. Using the parallel flow, we can assume that it is a horospherically concave hypersurface contained in the component $\mathcal C$ of $\h^{m+1} \setminus E(0)$ that contains ${\bf n}$.
\begin{figure}
  \centering
  \includegraphics[width=0.6\textwidth]{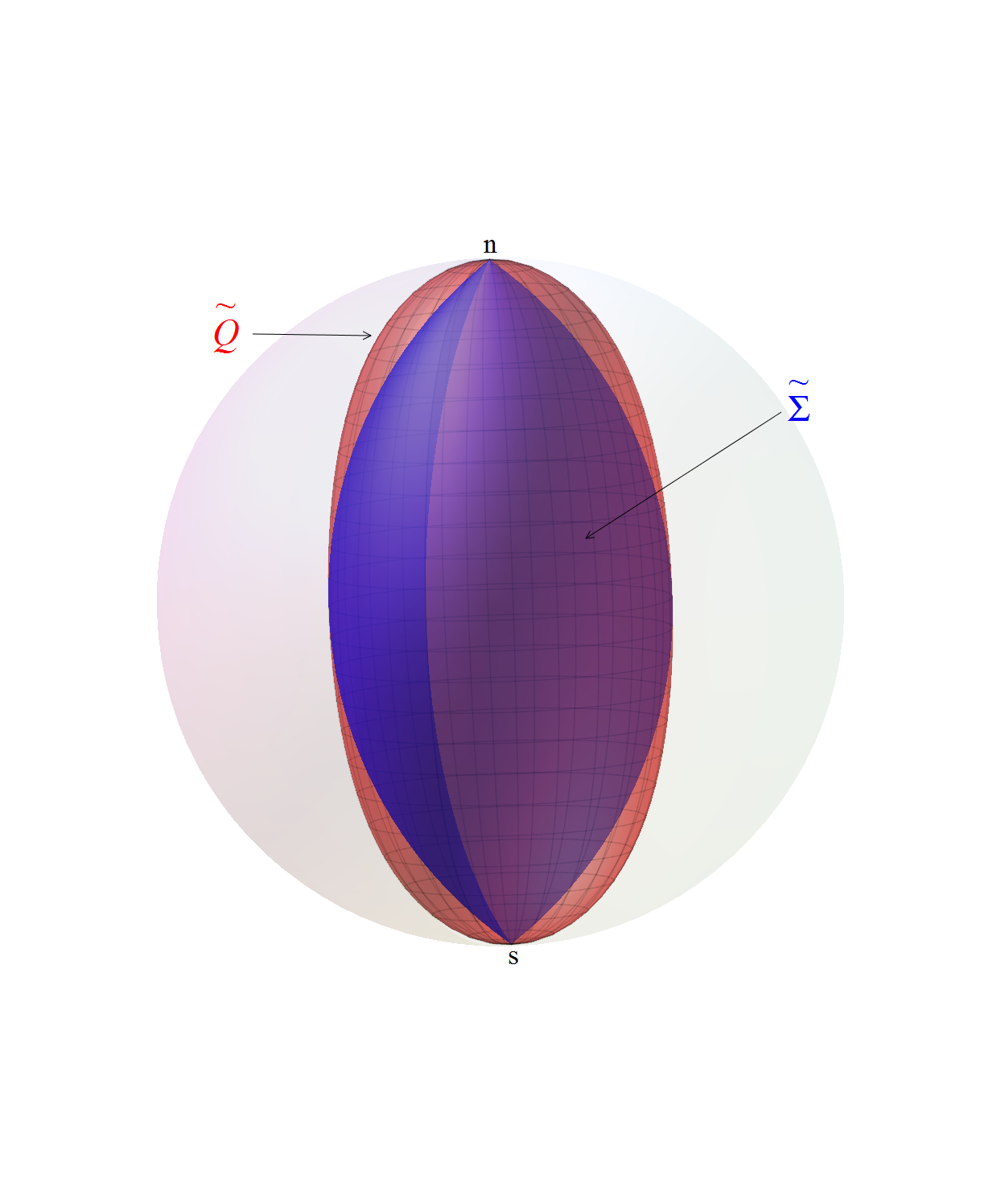}
  \caption{The first hypersurface $\tilde{Q}$ of a punctured solution \\
  that touches the hypesurface $\tilde{\Sigma}$.}\label{gr4d1}
\end{figure}

Let $g_{P}$ be a punctured solution of the problem \eqref{exequation}. Since $\partial \Gamma$ is a cone, the conformal metric $e^{2s}g_{P}$ is also a solution of the problem (\ref{exequation}), for every $s\in\R$.

There is $s_{0}>0$, such that for all $s\geq s_{0}$ the associated horospherical hypersurface $\phi_{s}$ associated to the  punctured solution $e^{2s}g_{P}$ is embedded and its interior is contained in the same component $\mathcal U$. The family of hypersufaces $\{ Q(s) \}_{s\geq s_{0}}=\{Im(\phi_{s})\}_{s\geq s_{0}}$ converges, in the Poincar\'{e} ball model, to the inclusion  $\ov{ \S^{m}_{+}}\setminus\{{\bf n }\} \hookrightarrow\S^{m}$.  Also, every $Q(s)$ admits a $C^{1}$-extension to ${\bf n}$ in the Poincar\'{e} ball model contained in $\R^{m+1}$ and their tangent hyperplanes at ${\bf n}$ are parallel to the hyperplane $x_{m+1}=0$.

There is a $t_{1}>0$ such that the associated hypersuface $\Sigma_{1}$ to $g_{1}=e^{2t_{1}}g$ intersects the family $\{ Q(s) \}_{s\geq s_{0}}$. Without loss of generality we assume that the associated hypersurface $\Sigma$ to $g$ intersects the family $\{ Q(s) \}_{s\geq s_{0}}$. Also, there is $s_{1}\geq s_{0}$ such that:
  \begin{enumerate}
    \item $Q(s_{1})\cap \Sigma\neq\emptyset$,
    \item for $s>s_{1}$ we have $Q(s)\cap\Sigma=\emptyset$.
  \end{enumerate}

Hence, we have found a first contact point. Observe that such contact point can not be at infinity, this follows since $\Sigma$ is contained in the interior of the equidistant to a geodesic and the family $Q(s)$ extends to ${\bf n}$. Now, if the first contact point is an interior point then $Q(s_1)$ and $\Sigma$ are tangent at such point. If the first contact point occurs on the boundary, $Q(s_{1})$ and $\Sigma$ are tangent at that point too, because all the hypersurfaces $Q(s)$ are orthogonal to the totally geodesic hypersurface $E(0)$ and $\Sigma$ is orthogonal to $E(0)$ too.

We reflect $Q(s_{1})$ and $\Sigma$ with respect to the totally geodesic hypersurface $E(0)$ and we obtain the horospherically concave hypersurfaces $\tilde{Q}$ and $\tilde{\Sigma}$ (cf. Figure \ref{gr4d1}). Their support functions are defined in $\s^{m}\setminus\{\bf n,s\}$. Let $\rho_{1}$ be the support function of $\tilde{Q}$ and $\tilde{\rho}$ be the support function of $\tilde{\Sigma}$. Since there is a contact point in the interior of $\tilde{Q}$, there is $0<\delta<\pi/2$ such that
\begin{equation*}
  \rho_{1}>\tilde{\rho} \text{ on }\partial\Omega,
\end{equation*}
where $\Omega=\{x\in\s^{m}:\delta<d_{\s^{m}}(x,{\bf n})<\pi-\delta\}$ and there exists  $x_{0}\in\Omega$ such that $\rho_{1}(x_{0})=\tilde{\rho}(x_{0})$, but this contradicts \cite[Proposition 3.1]{YLi}. This concludes the proof.
\end{proof}

\begin{figure}
  \centering
  \includegraphics[width=0.7\textwidth]{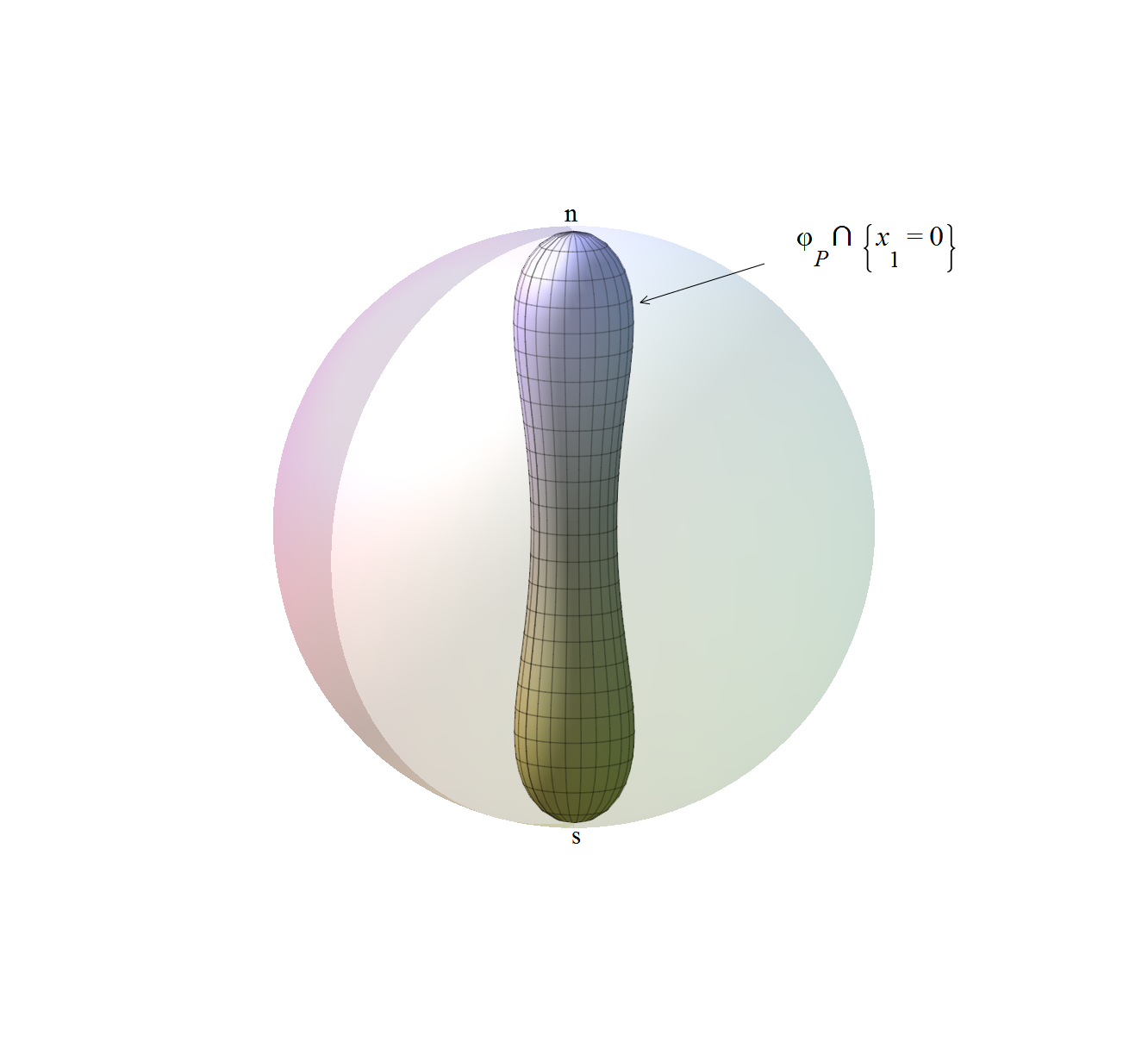}
  \caption{This is a slice of the hypersurface associated to $\sigma _k (\lambda _1 , \ldots ,\lambda _m) =cte$ using the Poincar\'{e} ball model for $m=3$ and $k=1$.}\label{equid}
\end{figure}

It is important to say that the $\sigma_{k}$-Yamabe problem in $\ov{ \S^{m}_{+} } \setminus\{ {\bf n} \}$ admits a punctured solution when $1\leq k <m/2$. That punctured solution is associated to
\begin{equation*}
 e^{-\rho (x)} = \sigma(x) =\left( (1+x_{m+1})^{\beta}+(1-x_{m+1})^{\beta} \right)^{\frac{1}{\beta}}, \quad x\in\ov{ \S^{m}_{+} }\setminus\{ {\bf n} \},
\end{equation*}
where $\beta=1-m/(2k)<0$, these solutions were constructed by S.-Y. A. Chang, Z. Han, and P. Yang \cite{Chang}. When $m$ is  even and $k=m/2$, the $\sigma_{k}$-Yamabe problem on the compact annulus has a solution  $g$ with $\sigma_{k}\left( \lambda(g) \right)=0$ and minimal boundary (cf. Figure \label{equid} when $m=3$ and $k=1$).

Also, it is good to say that the assumption on the existence of the punctured solution is not a necessary condition for the non-existence of solutions of degenerate problems in the compact annulus with minimal boundary. We have seen that punctured solutions are rotationally invariant (cf. Corollary \ref{Ch4:Punc}). Moreover, for the $\sigma_{k}$-Yamabe problem when $k>m/2$, there is no solution to the degenerate problem with minimal boundary and, also, there is no punctured solution of these problems (cf. \cite{Chang}).


\section{Noncompact annulus with boundary}\label{CompleteAnnulus}

Now we focus on different boundary conditions in the annulus. At one boundary component we will impose mild conditions on the metric and at the other we will impose constancy of the mean curvature of the conformal metric.

Our next result will say that any conformal metric $g=e^{2\rho}g_{0}$ in $$\mathbb{A}(r,\pi/2]:= \set{x \in \s ^m \, : \, \, r < d_{\s^m}(x,{\bf n}) \leq \pi /2}$$ satisfying certain property at its end and solution of a degenerate problem with non-negative constant mean curvature on its boundary, has unbounded Schouten tensor. In other words, we establish a non-existence result for degenerate (and non-degenerate) elliptic equations in $\mathbb{A}(r,\pi/2]$. Specifically,

\begin{theorem}\label{Ch4:Theo.4.3.6}
Let $r\in(0,\pi/2)$, $c\geq0$ be a non-negative constant and $g=e^{2\rho}g_{0}$ be a conformal metric in $\mathbb{ A }\left(r,\frac{\pi}{2}\right]$ that is solution of the following degenerate elliptic problem:
\begin{equation*}
  \left\{
               \begin{array}{cccc}
                 f(\lambda(g))       &=& 0 & \quad\text{in } \mathbb{A}(r,\pi/2],          \\
                 h(g) &=& c & \quad\text{on }\partial \s^{m}_{+}.
               \end{array}
      \right.
  \end{equation*}

If $e^{2\rho}+|\nabla\rho|^{2}:\mathbb{A}(r,\pi/2]\to\r$ is proper then $\lambda(g)$ is unbounded.
\end{theorem}
\begin{proof}
The proof is by contradiction. Suppose that $\lambda(g)$ is bounded. Using the parallel flow we can assume that $\varphi_{P}: \mathbb{A}(r,\pi/2] \to\b^{m+1} \subset \r ^{m+1}$ is a proper horospherically concave hypersurface. Recall that this property is invariant under the parallel flow.

Consider the continuous extension $\Phi:\ov{ \mathbb{A}(r) }\to\r^{m+1}$ of $\varphi_{P}:\mathbb{A}(r,\pi/2]\to \b^{m+1} \subset \r^{m+1}$, defined by
\begin{equation*}
  \Phi(x)=\left\{\begin{array}{cc}
                   \varphi_{P}(x) & x\in\mathbb{A}(r,\pi/2],\\
                   x & x\in S_{r}(\bf n).
                 \end{array}
             \right.
\end{equation*}

We consider the foliation of $\h^{m+1}$ by horospheres $\{H(s)\}_{s\in\r}$ having the north pole $\{\bf n\}$ as the boundary at infinity, $s$ is the signed distance between $H(s)$ and the origin of the Poincar\'{e} ball model. Since $r>0$, we have that there is $s_{1}\in\r$ such that $H(s_{1})\cap Im(\Phi)\neq\emptyset$ and $H(s)\cap Im(\Phi)=\emptyset$ for $s>s_{1}$.

\begin{figure}[!h]
  \includegraphics[width=0.6\textwidth]{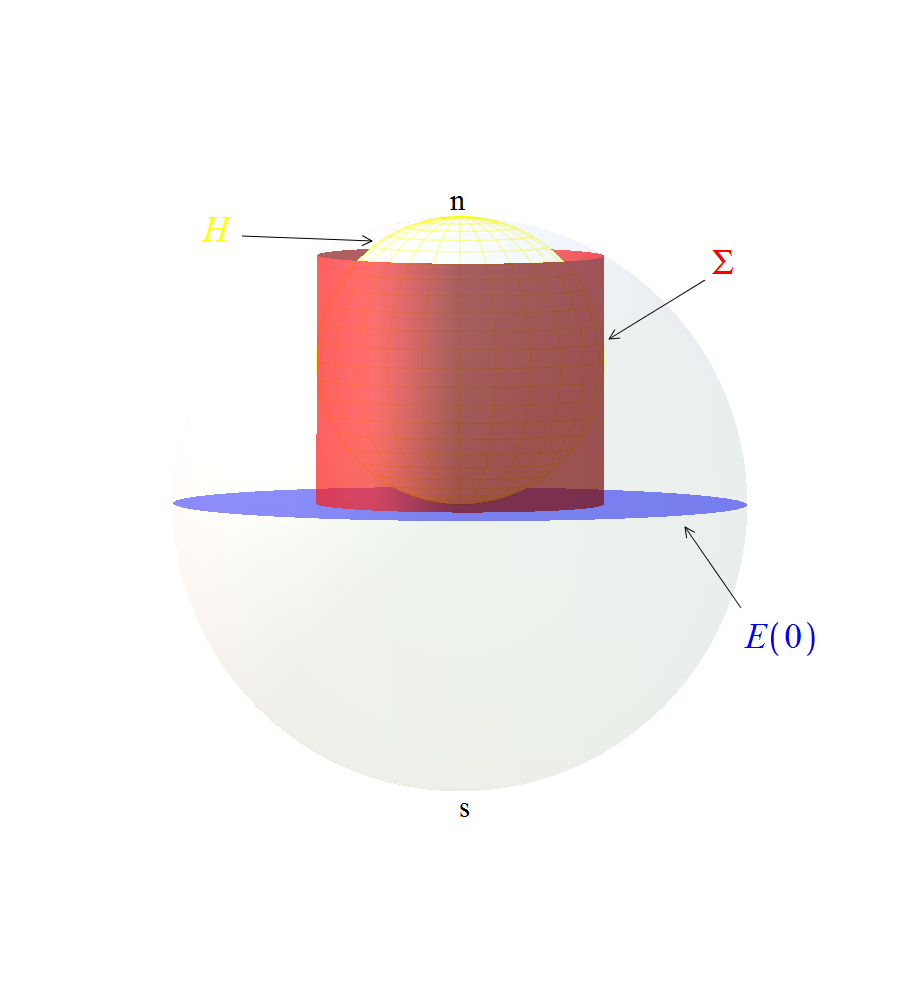}
  \caption{The horosphere $H$ touching the hypersurface $\Sigma$.}\label{gr4e1}
\end{figure}

Also, since $c\geq0$, the angle between $\Sigma=Im(\Phi)$ and the equidistant $E(-c)$ is acute and the angle between the horosphere $H=H(s_{1})$ and the equidistant is obtuse, hence the first contact point is at the interior of $\Sigma$ (cf. Figure \ref{gr4e1} when $c=0$). That is, there is $x\in \mathbb{A}(r)$ where $\phi_{P}(x)\in H$. Also $\phi_{P}^{-1}(H)\subset\mathbb{A}(r)$ is compact and there is $r<r_{1}<\pi/2$ such that
\begin{equation*}
  x\in\phi_{P}^{-1}(H)\subset\mathbb{A}(r_{1}).
\end{equation*}

Let $\rho_{0}$ be the support function associated to the horosphere $H$ restricted to $\ov{ \mathbb{A}(r_{1}) }$. Then, $\rho_{0}(x)=\rho(x)$ and
\begin{equation*}
  \rho>\rho_{0}\quad\text{on }\partial\mathbb{A}(r_{1}),
\end{equation*}
hence, by \cite[Proposition 3.1]{YLi}, $\rho(x)>\rho_{0}(x)$, which is a contradiction. That concludes the proof.
\end{proof}

A similar conclusion can be drawn for conformal metrics that are solutions of a non-degenerate elliptic problem that satisfy certain mild conditions.

\begin{theorem}\label{Ch4:Theo.4.3.7}
Let $0<r<\pi/2$, $c\in\r$ be a constant and $g=e^{2\rho}g_{0}$ be a conformal metric in $\mathbb{ A }\left(r,\frac{\pi}{2}\right]$ that is solution of the following non-degenerate elliptic problem:
\begin{equation*}
  \left\{
               \begin{array}{cccc}
                 f(\lambda(g))       &=& 1 & \quad\text{in } \mathbb{A}(r,\pi/2],          \\
                 h(g) &=& c & \quad\text{on }\partial \s^{m}_{+}, \\
                 \displaystyle{\lim_{x\to q}\rho(x)} & = & +\infty & \quad\forall\: q\in\partial B_{r}({\bf n}).
               \end{array}
      \right.
  \end{equation*}

Set $\sigma=e^{-\rho}$, if $|\nabla\sigma|^{2}$ is Lipschitz then $\nabla^{2} ( \sigma^{2} )$ is unbounded.
\end{theorem}

\begin{proof}
The proof is by contradiction. We suppose that $\nabla^{2}\sigma^{2}$ is bounded. Using the parallel flow, we can assume that $\phi:\mathbb{A}(r,\pi/2]\to\h^{m+1}$ is a properly embedded horospherically concave hypersurface. Since $h(g)=c$, we have that the boundary $\partial\Sigma$ is contained in $E(-c)$.

Take a closed ball $Q$ centered at the origin of the Poincar\'{e} model of radius big enough so that $\partial\Sigma$ is in the interior of $Q$. Since $f$ is homogeneous of degree one and $f(1,\ldots,1)>0$, there is a constant $\lambda_{0}>0$ such that
\begin{equation*}
  f(\lambda_{0},\ldots,\lambda_{0})=1,
\end{equation*}
and using the parallel flow, we can assume that $0<\lambda_{0}<1/2$.

We work in the Poincar\'{e} ball model. Consider the family of totally umbilic spheres in the Hyperbolic space centered at the $x_{m+1}-$axis, $\{ Z(s)\}_{s\in(-1,1)}$, such that the principal curvatures are equal to
  \begin{equation*}
    k_{0}=\frac{1+2\lambda_{0}}{1-2\lambda_{0}}>1.
  \end{equation*}

Observe that the support function of all these totally umbilic spheres are solutions of the same non-degenerate elliptic problem.

Consider the continuous extension $\Phi:\ov{ \mathbb{A}(r) }\to\r^{m+1}$ of $\varphi_{P}:\mathbb{A}(r,\pi/2]\to\r^{m+1}$, defined by
\begin{equation*}
  \Phi(x)=\left\{\begin{array}{cc}
                   \varphi_{P}(x) & x\in\mathbb{A}(r,\pi/2],\\
                   x & x\in S_{r}(\bf n).
                 \end{array}
             \right.
\end{equation*}

Since $r>0$, there is a $\delta>0$ such that for all $s\in(1-\delta,1)$:
  \begin{enumerate}
  \item $Z(s)\cap\Sigma=\emptyset$,
  \item $Z(s)\cap Q = \emptyset$\label{ob2}.
  \end{enumerate}

We take one of them, say $Z_{0}=Z(s)$. In the Poincar\'{e} ball model,  consider the circle centered at the origin ${\bf 0}$ of radius $s$ and passing through the center of $Z_0$.  Move $Z_{0}$ along this circle until we have the first totally umbilic hypersurface $\tilde Z _0$ touching the hypersurface. By item \ref{ob2}, the contact point is at the interior (cf. Figure \ref{gr4f1} when $c=0$). That is, there is $x\in\mathbb{A}(r)$ such that $\varphi_{P}(x)\in \tilde Z_{0}$. At such contact point, the canonical orientations of $\Sigma $ and $\tilde Z_0$ agree.

\begin{figure}[!h]
  \includegraphics[width=0.8\textwidth]{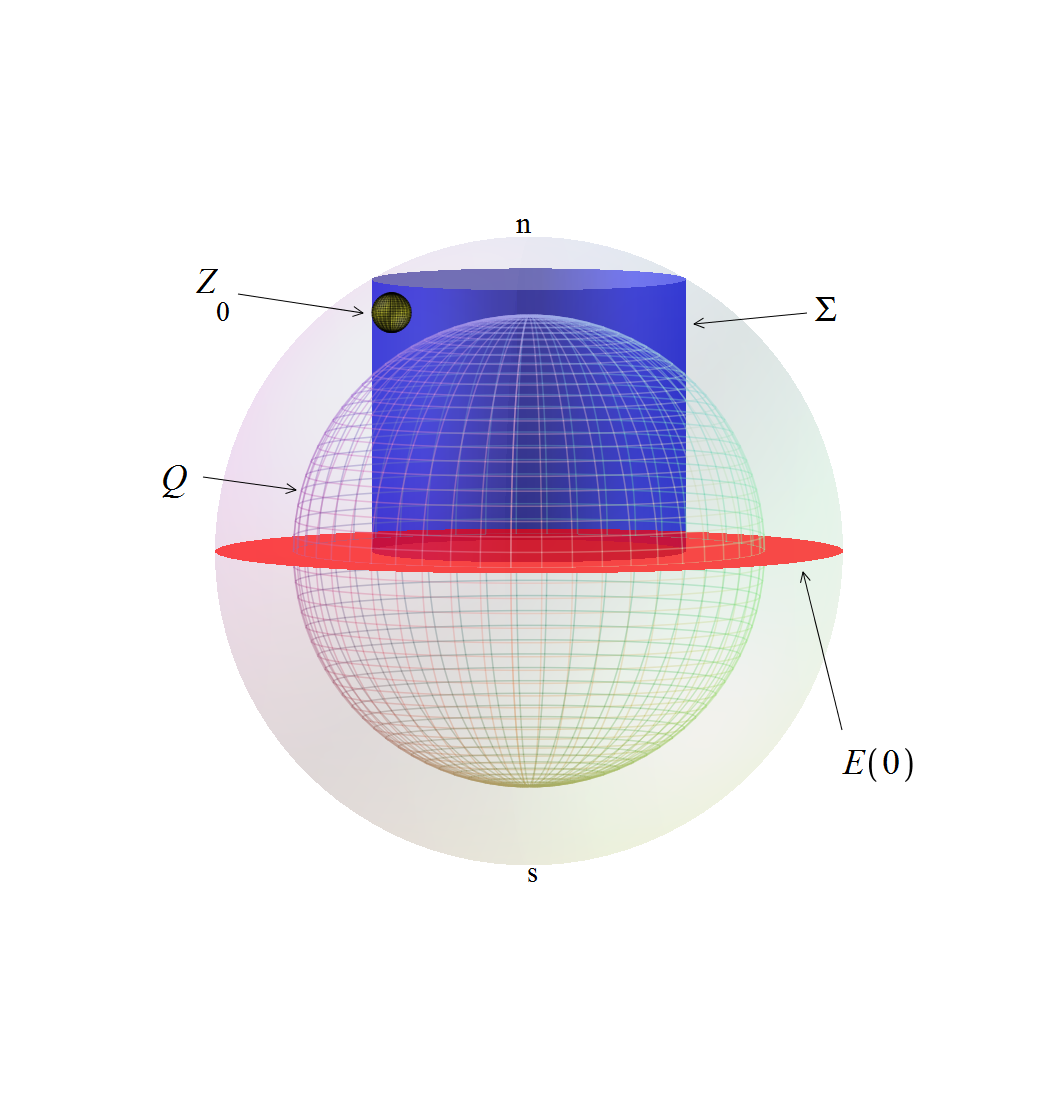}
  \caption{$\tilde Z_{0}$ touching at the interior of $\Sigma$, case $c=0$.}\label{gr4f1}
\end{figure}

Let $\rho_{0}$ be the support function of $Z_{0}$ restricted to $\mathbb{A}(r)$, then we have that
 \begin{equation*}
   \rho\geq\rho_{0} \text{ on }\mathbb{A}(r)\quad\text{ and }\quad\rho(x)=\rho_{0}(x),
 \end{equation*}
 then by the strong maximum principle, $\rho=\rho_{0}$. So, $\Sigma$ is part of a sphere, but $\Sigma$ has non-empty ideal boundary. This is contradiction, which concludes the proof.
\end{proof}


\section{The 2-dimensional case}\label{2dim}

As we have seen, the Schouten tensor is defined for Riemannian manifolds $(\man^{m},g_{0})$ when $m \geq 3$. Let us consider the conformal metric $g=e^{2\rho}g_{0}$, where $\rho\in C^{\infty}(\man)$, then we have the following relation:
\begin{equation*}
  \textrm{Sch}(g) +\nabla^{2}\rho+\frac{1}{2}|\nabla\rho|^{2}g_{0}=\textrm{Sch}(g_{0})+\nabla\rho\otimes\nabla\rho,
\end{equation*}
where $\nabla$ ,$\nabla ^{2}$ are the gradient and the hessian with respect the metric $g_{0}$ respectively, and $\abs{\cdot}$ the norm with respect of $g_{0}$.

In the case of the standard sphere $\left(\s^{m},g_{0} \right)$, we know that  $\textrm{Sch}(g_{0})=\frac{1}{2}g_{0}$, then for every conformal metric $g=e^{2\rho}g_{0}$, we have that
\begin{equation}\label{Ch4:Conf}
  \textrm{Sch}(g) +\nabla^{2}\rho+\frac{1}{2}|\nabla\rho|^{2}g_{0}=\frac{1}{2}g_{0}+\nabla\rho\otimes\nabla\rho .
\end{equation}

So, we can take the above expression as a definition of the Schouten tensor for a conformal metric to the standard one in domains of the sphere $\s^{2}$. Hence, we can consider Yamabe type problems in $\s^{2}$. The Yamabe problem is equivalent to
 \begin{equation*}
  \lambda_{1}+\lambda_{2}=\frac{1}{2} \text{ in } \S^{2},
\end{equation*}where $\lambda _i$, $i=1,2$, are the eigenvalues of the Schouten tensor given by \eqref{Ch4:Conf}. In other words, the conformal metric $g=e^{2\rho}g_{0}$ in $\S^{2}$ has constant scalar curvature $R(g)=1$ or, equivalently, constant Gaussian curvature, since from (\ref{Ch4:Conf})
\begin{equation*}
  \textrm{Tr}\left( g^{-1}\textrm{Sch}(g) \right)=e^{-2\rho}\left( 1-\Delta\rho \right)=\frac{1}{2}R(g) = K ,
\end{equation*}where $K$ is the Gaussian curvature of $g= e^{2\rho} g_{0}$, i.e., the Yamabe Problem reduces to the Liouville Problem.

This example says that the definition of the Schouten tensor for conformal metrics w.r.t. the standard metric in domains of the sphere $\s^{2}$, given by \eqref{Ch4:Conf}, makes sense. Then, we can consider more general elliptic problems for conformal metrics in $\s^{2}$.

We establish the analogous result we can obtain in the case of domains of $\s^{2}$ without proof. First, for geodesic disk we have:

\begin{theorem}
 Let $ (f , \Gamma )$ be a degenerate elliptic data for conformal metrics and let $c\leq 0$ be a constant. Then, there is no  conformal metric $g=e^{2\rho}g_{0}$ in $\overline{\mathbb{S}^{2}_{+} }$, $\rho\in C^{2,\alpha}\left(\overline{ \mathbb{S}^{2}_{+} }\right)$, satisfying
    \begin{equation*}
      \left\{
               \begin{array}{ccccl}
                 f( \lambda(g) )       &=& 0 & \text{ in } & \mathbb{S}^{2}_{+},          \\
                 h(g) &=& c & \text{ on } & \partial \mathbb{S}^{2}_{+},
               \end{array}
      \right.
    \end{equation*}where $\lambda(g)=(\lambda_{1},\lambda_{2})$ is composed by the eigenvalues of the Schouten tensor of the metric $g=e^{2\rho}g_{0}$.
\end{theorem}

Second, for compact annulus, we have the following non-existence result:

\begin{theorem}\label{Ch4:Theo.4.5.1}
If the problem (\ref{exequation}) with $m=2$ admits a punctured solution, then there is no solution of the following degenerate elliptic problem:
  \begin{equation*}
     \left\{
               \begin{array}{cccc}
                 f( \lambda(g) )       &=& 0 & \quad\text{in } \mathbb{A}(r),          \\
                 h(g) &=& 0 & \quad\text{on }\partial \mathbb{A}(r),
               \end{array}
      \right.
  \end{equation*}
  where $\lambda(g)=(\lambda_{1},\lambda_{2})$ is composed by the eigenvalues of the Schouten tensor of $g$.
\end{theorem}

In this part, it is good to say that it is possible that the punctured solution in Theorem \ref{Ch4:Theo.4.5.1} might not exist. For example, the Yamabe problem, or Liouville Problem in the annulus $\ov{\mathbb{A}(r)}\subset \S^{2}$, $0<r<\pi/2$, has a solution with zero scalar curvature and minimal boundary, then there is no punctured solution for the Yamabe problem on $\ov{ \s^{2}_{+} }\setminus\{ {\bf n} \}$.

The solution of that problem is given by the conformal metric $g=e^{2\rho}g_{0}$ in $\ov{\mathbb{A}(r)}$, $0<r<\pi/2$, where (cf. Figure \ref{equid2})
\begin{equation*}
   e^{2\rho(x,y,z)}=\frac{1}{ \sigma^{2}(x,y,z) }=\frac{1}{ 1-z^{2} } \quad\text{for all }(x,y,z)\in \ov{\mathbb{A}(r)}.
\end{equation*}

\begin{figure}[!h]
  \centering
  \includegraphics[width=0.5\textwidth]{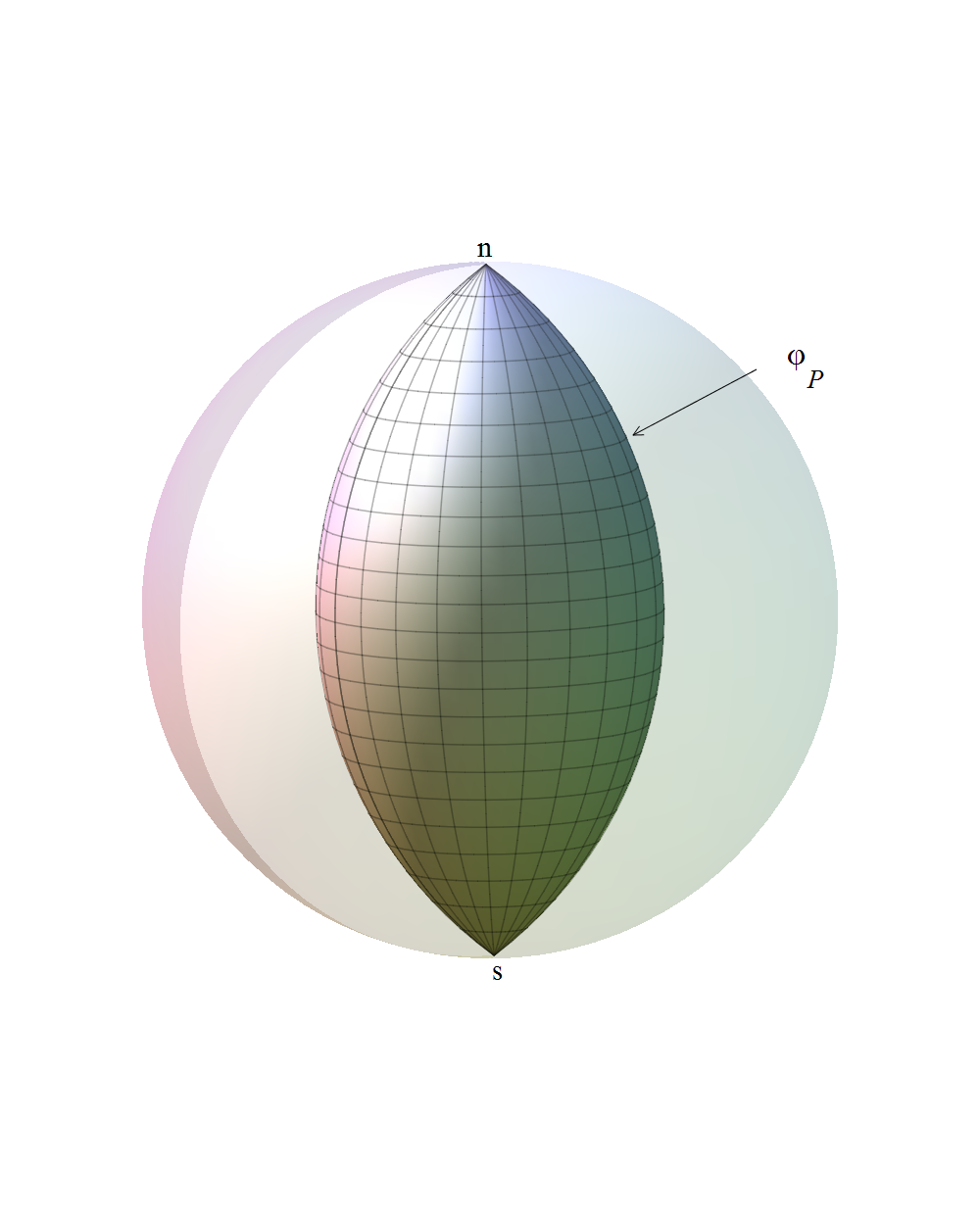}
  \caption{Surface associated to $\sigma=\frac{1}{3}\sqrt{1-z^2}$ in the Poincar\'{e} ball model.}\label{equid2}
\end{figure}

Also, in dimension $m>2$, we can define the conformal metric $g=e^{2\rho}g_{0}$ in $\ov{\mathbb{A}(r)}$, $0<r<\pi/2$, analogously, i.e.,
\begin{equation*}
   e^{2\rho(x_{1},\ldots,x_{m+1})}=\frac{1}{ \sigma^{2}(x_{1},\ldots,x_{m+1}) }=\frac{1}{ 1-x_{m+1}^{2} } \quad\text{for all }(x_{1},\ldots,x_{m+1})\in \ov{\mathbb{A}(r)},
\end{equation*}
but this conformal metric has constant scalar curvature equals to $(m-1)(m-2)>0$. When $m$ is even and $k=m/2$, this conformal metric is a solution of the degenerate $\sigma_{k}$-Yamabe problem in the compact annulus $\ov{ \mathbb{A}(r) }$ with minimal boundary. In the case of the complete annulus with boundary, we have

\begin{theorem}\label{Ch4:Theo.4.5.2}
Let $r\in(0,\pi/2)$, $c\geq0$ be a non-positive constant and $g=e^{2\rho}g_{0}$ be a conformal metric in $\mathbb{ A }\left(r,\frac{\pi}{2}\right)$ that is solution of the following degenerate elliptic problem:
\begin{equation*}
  \left\{
               \begin{array}{cccc}
                 f(\lambda(g))       &=& 0 & \quad\text{in } \mathbb{A}(r,\pi/2),          \\
                 h(g) &=& c & \quad\text{on }\partial \s^{m}_{+},
               \end{array}
      \right.
  \end{equation*}
  If $e^{2\rho}+|\nabla\rho|^{2}:\mathbb{A}(r,\pi/2]\to\r$ is proper then $\lambda(g)$ is not bounded.
\end{theorem}


\bibliographystyle{amsplain}

\begin{thebibliography}{10}

 \bibitem{Aub} T. Aubin, \textit{Equations différentielles non linéaires et problème de Yamabe concernant la courbure scalaire}, J. Math. Pures Appl. \textbf{55} (1976), 269-296.

\bibitem{BEQ} V. Bonini, J. M. Espinar, J. Qing,  \newblock{\it Hypersurfaces in Hyperbolic Space with Support Function},  \newblock Adv. in Math., {\bf 280} (2015), 506--548.

\bibitem{BEQ2} V. Bonini, J. M. Espinar, Jie Qing, {\it Hypersurfaces in Hyperbolic Poincar\'{e} Manifolds and Conformally Invariant PDEs}, Proc. A.M.S., {\bf 138}  (2010) {\bf no. 11}, 4109--4117.

\bibitem{SBreFMarANev10} S. Brendle, F.C. Marques, A. Neves, {\it Scalar curvature rigidity of geodesic balls in $\s ^n$}, Invent. Math., {\bf 63} (2010), 1237--1247.

\bibitem{Bra} H.L. Bray, \textit{The Penrose inequality in general relativity and volume comparison theorems involving scalar curvature}, Dissertation, Stanford University, 1997.


\bibitem{CE} M.P. Cavalcante, J.M. Espinar, {\it Uniqueness Theorems for Fully nonlinear Conformal Equations on Subdomains of the Sphere}. Preprint. Available online at: \texttt{http://arxiv.org/abs/1505.00733}.



\bibitem{Chang} S.-Y. A. Chang, Z. Han, and P. Yang, \textit{Classification of singular radial solutions to the $\sigma_{k}$ Yamabe equation on annular domains},  J. Differential Equations \textbf{216} (2005), \textbf{no. 2}, 482–501.




\bibitem{Esc} J. Escobar, \textit{Uniqueness Theorems on Conformal Deformation of Metrics, Sobolev Inequalities, and an Eigenvalue Estimate}, Comm. Pure Appl. Math. \textbf{43} (1990), \textbf{no. 7}, 857--883.

\bibitem{Esc2} J. Escobar, \textit{The Yamabe problem on manifolds with boundary}, J. Differential Geom. \textbf{35} (1992),  \textbf{no. 1}, 21–84.

\bibitem{Esc3} J. Escobar, \textit{Conformal deformation of a Riemannian metric to a constant scalar curvature metric with constant mean curvature on the boundary}, Indiana Univ. Math. J.  \textbf{45} (1996), \textbf{no. 4}, 917–943.

\bibitem{Esc4} J. Escobar, \textit{Conformal deformation of a Riemannian metric to a scalar flat metric with constant mean curvature on the boundary}, Ann. of Math. (2)  \textbf{136}  (1992),  \textbf{no. 1}, 1–50.

\bibitem{Esp} J.M. Espinar, {\it Invariant Conformal Metrics on $\s^{n}$}, Trans. Amer. Math. Soc., {\bf 363} (2011) {\bf no. 11}, 5649--5662.

\bibitem{EGM} J. M. Espinar, J. A. G\'alvez, P. Mira, \newblock{\it Hypersurfaces in $\mathbb H^{n+1}$ and conformally invariant equations: the generalized Christoffel and Nirenberg problems.} J. Eur. Math. Soc.  {\bf 11} (2009), \textbf{no. 4}, 903--939.


\bibitem{Fi1} W. J. Firey, \textit{The determination of convex bodies from their mean radius of curvature functions}, Mathematika  \textbf{14}  (1967), 1--13.

\bibitem{Fi2} W. J. Firey, \textit{Christoffel's problem for general convex bodies}, Mathematika  \textbf{15}  (1968), 7--21.

\bibitem{FukNak} K. Fukui, T. Nakamura, {\it A topological property of Lipschitz mappings}, Topology and its Applications {\bf 148} (2005), 143--152.


\bibitem{Gurs}  M. J. Gursky, J. A. Viaclovsky, \textit{Volume comparison and the $\sigma_{k}$-Yamabe problem}, Adv. Math. \textbf{187} (2004), \textbf{no. 2}, 447–487.

\bibitem{HaWa} F. Hang, X. Wang, {\it A new approach to some nonlinear geometric equations in dimension two}, Calc. Var. Partial Diff. Eq., {\bf 26} (2006), 119--135.

\bibitem{J} A. Jim\'{e}nez,  {\it The Liouville equation in an annulus}, J. Nonlinear Anal., {\bf 75} (2012), 2090--2097.




\bibitem{YLi06} Y.Y. Li,\textit{Conformally invariant fully nonlinear elliptic equations and isolated singularities}, J. Funct. Anal. \textbf{233} (2006), \textbf{no. 2}, 380--425.

\bibitem{YLi} Y.Y. Li, \textit{Degenerate conformally invariant fully nonlinear elliptic equations}, Arch. Ration. Mech. Anal.  \textbf{186}  (2007),  \textbf{no. 1}, 25–-51.

\bibitem{LiLi1} A. Li, Y.Y. Li, {\it On some conformally invariant fully nonlinear equations}, Comm. Pure Appl. Math., {\bf 56} (2003), 1416--1464.

\bibitem{LiLi2} A. Li, Y.Y. Li, {\it On some conformally invariant fully nonlinear equations. II. Liouville, Harnack and Yamabe}, Acta Math., {\bf 195} (2005), 117--154.

\bibitem{LiLi3} A. Li, Y.Y. Li, {\it A fully nonlinear version of the Yamabe problem on manifolds with boundary}, J. Eur. Math. Soc. (JEMS) {\bf 8} (2006) {\bf no. 2}, 295--316.

\bibitem{YLiLNgu14} Y.Y. Li, L. Nguyen, {\it  A fully nonlinear version of the Yamabe problem on locally conformally flat manifolds with umbilic boundary}, Adv. Math., {\bf 251} (2014), 87--110.


\bibitem{Ma} F. C. Marques, \textit{Existence results for the Yamabe problem on manifolds with boundary}, Indiana Univ. Math. J. \textbf{54}  (2005) \textbf{no. 6}, 1599–1620.


\bibitem{Sheng} W. M. Sheng, N. S. Trudinger, X.-J. Wang, \textit{The k-Yamabe problem}. Int. Press, Boston, MA. Surv. Differ. Geom.  \textbf{XVII} (2012) \textbf{17}, 427–457.

\bibitem{Sheng2} W. M. Sheng, N. S. Trudinger and X.-J. Wang, \textit{The Yamabe problem for higher order curvatures}, J. Diff. Geom. \textbf{77} (2007), 515–553.


\bibitem{Scho}  R. Schoen, \textit{Conformal deformation of a Riemannian metric to constant scalar curvature}, J. Diff. Geom. \textbf{20} (1984), 479--495.

\bibitem{SchoenYau}  R. Schoen, S.T. Yau, \textit{Conformally flat manifolds, Kleinian groups and scalar curvature}, Invent. Math., \textbf{92} (1988), {\bf no. 1}, 47--71.

\bibitem{Schwa} F. Schwartz, \textit{The zero scalar curvature Yamabe problem on noncompact manifolds with boundary}, Indiana Univ. Math. J.  \textbf{55}  (2006) \textbf{no. 4}, 1449–1459.

\bibitem{Spie} F. M. Spiegel, \textit{Scalar curvature rigidity for locally conformally flat manifolds with boundary}. Preprint. Available online at: \texttt{http://arxiv.org/abs/1511.06270}.


\bibitem{Tru} N. Trudinger, \textit{Remarks concerning the conformal deformation of Riemannian structures on compact manifolds}, Ann. Scuola Norm. Sup. Pisa \textbf{22} (1968), 265-274.

\bibitem{Yam} H. Yamabe, {\it On a deformation of Riemannian structures on compact manifolds}, Osaka Math. J. \textbf{12} (1960), 21-37.

\bibitem{Zhi} J. Zhiren, {\it A counterexample of the Yamabe problem for complete noncompact manifolds},
Lect. Notes Math. \textbf{1036} (1988), 93--101.

\end{thebibliography}

\end{document}